\pdfoutput=1
\documentclass[reqno]{amsart}

\usepackage[letterpaper,margin=1.25in]{geometry}

\usepackage{amsmath,amsthm,amssymb,amscd}
\usepackage[all]{xy}
\newdir{ >}{{}*!/-8pt/@{>}} %

\usepackage{upgreek} %

\usepackage{xcolor}
\definecolor{darkgreen}{rgb}{0,0.45,0}
\usepackage{ifpdf}
\ifpdf  \usepackage[pdftex,colorlinks,urlcolor=blue,citecolor=darkgreen,linkcolor=darkgreen,linktocpage]{hyperref}
\else   \usepackage[dvipdfmx,colorlinks,urlcolor=blue,citecolor=darkgreen,linkcolor=darkgreen,linktocpage]{hyperref}
\fi

\newtheorem{de}{Definition}[section]
\newtheorem{lem}[de]{Lemma}
\newtheorem{prop}[de]{Proposition}

\newtheorem{thm}[de]{Theorem}
\newtheorem{thmintro}{Theorem}

\theoremstyle{remark}
\newtheorem{rem}[de]{Remark}
\newtheorem{ex}[de]{Example}

\newcommand{\cxymatrix}[1]{\vcenter{\xymatrix{#1}}}

\newcommand{\N}{\ensuremath{\mathbb{N}}}
\newcommand{\Z}{\ensuremath{\mathbb{Z}}}

\newcommand{\R}{\ensuremath{\mathbb{R}}}

\newcommand{\Diff}{\mathrm{Diff}}
\newcommand{\Plot}{\mathrm{Plot}}
\newcommand{\Vect}{\mathrm{Vect}}
\newcommand{\dvs}{\mathrm{dvs}}
\newcommand{\DVPB}{\mathrm{DVPB}}

\newcommand{\Hom}{\mathrm{Hom}}
\newcommand{\Supp}{\mathrm{Supp}}
\newcommand{\Mor}{\mathrm{Mor}}
\newcommand{\op}{\mathrm{op}}
\DeclareMathOperator{\colim}{colim}
\newcommand{\Obj}{\mathrm{Obj}}
\newcommand{\Ab}{\mathrm{Ab}}
\newcommand{\ddt}{\frac{d}{dt}}
\newcommand{\ddx}{\frac{d}{dx}}
\newcommand{\ddy}{\frac{d}{dy}}
\newcommand{\cC}{\mathcal{C}}
\newcommand{\cI}{\mathcal{I}}

\newcommand{\lra}{\longrightarrow}
\newcommand{\llra}[1]{\stackrel{#1}{\lra}}  %

\newcommand{\Exterior}{\mathchoice{{\textstyle\bigwedge}}%
{{\bigwedge}}%
{{\textstyle\wedge}}%
{{\scriptstyle\wedge}}}

\title{Exterior bundles in diffeology}

\author{J. Daniel Christensen}
\email{jdc@uwo.ca}
\address{Department of Mathematics, University of Western Ontario, London, Ontario, Canada}

\author{Enxin Wu}
\email{exwu@stu.edu.cn}
\address{Department of Mathematics, Shantou University, Guangdong, P.R. China}

\date{July 9, 2021}

\begin{document}

\begin{abstract}
We explore several notions of $k$-form at a point in a diffeological space,
construct bundles of such $k$-forms, and compare sections of these
bundles to differential forms.
As they are defined locally, our $k$-forms can contain more information
than the values of differential forms contain, and we illustrate this with
many examples.
To organize our work, we develop the basic theory of diffeological vector
pseudo-bundles, including a detailed understanding of their limits and colimits,
as well as a variety of fibrewise operations such as products, direct sums,
tensor products, exterior powers and dual bundles.
\end{abstract}

\makeatletter
\@namedef{subjclassname@2020}{%
  \textup{2020} Mathematics Subject Classification}
\makeatother


\maketitle

\tableofcontents

\section{Introduction}

Diffeology is a framework for studying general smooth spaces, including smooth manifolds, 
singular spaces (such as smooth manifolds with boundary or corners, orbifolds, 
irrational tori and other exotic quotients, etc.), and infinite-dimensional 
spaces (e.g., function spaces between smooth manifolds, diffeomorphism groups, etc.)
It is natural to try to extend ideas from differential geometry to diffeological spaces, 
possibly under some mild assumptions, unifying similar constructions made in special cases.
In the current paper, we focus on the connection between
tangent bundles and differential forms on diffeological spaces.

The notion of a differential form on a diffeological space
goes back to early work of Souriau~\cite{So3}
and has been studied by many authors, such as~\cite{GI,I2,La}:
a $k$-form $\omega$ on a diffeological space $X$ is defined to be a compatible
family of $k$-forms $\omega_p \in \Omega^k(U)$ for each plot $p : U \to X$.
In classical differential geometry, a $k$-form can be defined as a section
of a bundle of exterior $k$-forms, and one of the goals of the present paper
is to investigate various ways that this approach can be used for diffeological spaces.

In~\cite{I2}, Iglesias-Zemmour gives one solution to this problem.
He defines a space $\Uplambda^k_x(X)$ of $k$-forms at the point $x \in X$
and assembles these into a diffeological vector pseudo-bundle (Definition~\ref{de:dvpb})
$\Uplambda^k X$ over $X$ whose sections recover the $k$-forms on $X$.

We introduce two new bundles over $X$.
The bundle $(\Exterior^k T^{\dvs} X)^*$ is formed just as in classical differential
geometry, by starting with the (internal) tangent bundle $T^{\dvs} X$ of $X$, and taking
the dual of the exterior power.
Here the tangent bundle is the colimit of the ordinary tangent
bundles $TU \to U$ in the category of diffeological vector pseudo-bundles,
where the colimit is taken over plots $p : U \to X$.

In order to compare sections of $(\Exterior^k T^{\dvs} X)^*$ with ordinary differential forms,
we also introduce another bundle $(T^k X)^*$, the dual of a bundle $T^k X$
whose definition is inspired by the tangent bundle:
$T^k X$ is the colimit of the ordinary exterior bundles $\Exterior^k TU \to U$
over the category of plots in $X$.

The three bundles are related by natural smooth bundle maps
\[
  \xymatrix{
  \Uplambda^k X \ar@{ >->}[r]^i & (T^k X)^* & (\Exterior^k T^{\dvs} X)^* \ar[l]_{\rho^*} ,
  }
\]
with $i$ being injective and $\rho^*$ being the dual of a map $\rho$
going in the other direction. We don't know whether or not $i$ is always an induction.
The results of Section~\ref{se:duality} can be summarized as follows.

\begin{thmintro}\label{th:intro}
For any diffeological space $X$ and any $k \geq 0$, the map $i$ induces
an isomorphism on sections, so that in particular we have a natural isomorphism
\[
  \Gamma((T^kX)^* \to X) \lra \Omega^k(X)
\]
of diffeological vector spaces.

If $k = 0$ or $1$, or $X$ is filtered, then the map $\rho$ is an isomorphism
of diffeological vector pseudo-bundles.
In particular, $\rho^*$ induces an isomorphism on sections.
\end{thmintro}

Filtered diffeological spaces include smooth manifolds, irrational tori
and all fine diffeological vector spaces.

The inclusion $i$ can be strict, as we show in Section~\ref{se:spaghetti}.
The vector space $\Uplambda^k_x(X)$ is defined to be a certain quotient of $\Omega^k(X)$,
so by definition, it only contains $k$-forms at the point $x$ which extend globally
to smooth $k$-forms on $X$.
In contrast, an element of $(T^k_x X)^*$ is specified by simply giving compatible
$k$-forms at $0 \in U$ for each pointed plot $(U,0) \to (X,x)$,
subject to a smoothness condition.
This allows for forms that reflect the local geometry, but cannot be extended.

Similarly, the elements of $(\Exterior^k T^{\dvs} X)^*$ allow further flexibility
in that they can be detected by exterior products of tangent vectors
coming from different plots.
We denote the sections of $(\Exterior^k T^{\dvs} X)^*$ by $\widetilde{\Omega}^k(X)$.
When $k \geq 2$ and $X$ is not filtered, this can differ from $\Omega^k(X)$.
We give a simple example of this in the introduction to Section~\ref{se:duality},
which also shows that these new forms can encode useful information.

\medskip

We now give a summary of the contents of the paper.
In Section~\ref{se:background} we give background about diffeological spaces,
tangent bundles, differential forms, and the bundle $\Uplambda^k X$.
Section~\ref{se:operations} is the technical core of the paper.
We study diffeological vector pseudo-bundles,
proving results about limits and colimits of diffeological vector pseudo-bundles.
We describe analogs of all of the standard fibrewise bundle operations in this setting,
extending work of~\cite{V} and~\cite{P}.
For example, we study products, direct sums, tensor products, exterior powers
and $\Hom$ bundles.
This material will be of use in any setting that involves diffeological vector
pseudo-bundles.
In Section~\ref{se:duality}, we introduce our new bundles and
prove the results summarized in Theorem~\ref{th:intro}.
The remaining sections give detailed calculations of the various types of
forms, showing that the claims of Theorem~\ref{th:intro} cannot be strengthened.
Section~\ref{se:axes} discusses the axes in $\R^2$, with the gluing diffeology,
while Section~\ref{se:spaghetti} discusses $\R^2$ with the spaghetti diffeology.

We would like to thank the anonymous referee for carefully reading the manuscript
and providing valuable suggestions for improving the presentation.

\section{Background}\label{se:background}

\subsection{Diffeological spaces}\label{subsection-diff}

Diffeological spaces are generalizations of smooth manifolds which
include singular spaces and infinite-dimensional spaces.
The idea was first introduced in~\cite{S},
and further developed in~\cite{I1,I2}. \cite{I2} is currently the standard textbook for this subject. We
summarize the basic concepts and results in this section. See~\cite{I2} or~\cite[Section~2]{CSW} for details.

\begin{de}
 A \textbf{diffeological space} is a set $X$ together with a family of maps $U \to X$ (called \textbf{plots}) for each
 open subset $U$ of $\R^n$ for $n \in \N$, such that the following conditions hold:
 \begin{enumerate}
  \item every constant map is a plot;
  \item if $U \to X$ is a plot and $V \to U$ is a smooth map between open subsets
  of Euclidean spaces, then the composite $V \to U \to X$ is a plot;
  \item a map is a plot if and only if it is locally a plot (with respect to the usual open covers).
 \end{enumerate}

 A map between diffeological spaces is called \textbf{smooth} if it sends plots of the domain to the plots of the
 codomain.

Diffeological spaces together with smooth maps form a category denoted $\Diff$. An isomorphism in $\Diff$
 is called a \textbf{diffeomorphism}.
\end{de}

Here are the nice properties of the category $\Diff$:
\pagebreak[2] %

\begin{thm}\
 \begin{enumerate}
  \item The category of smooth manifolds and smooth maps is a full subcategory of $\Diff$.
  \item The category $\Diff$ is complete, cocomplete and cartesian closed.  \qed
 \end{enumerate}
\end{thm}

The descriptions of limits, colimits and function spaces are quite
simple, and are concisely described in~\cite[Section~2]{CSW}.

A smooth map $i : X \to Y$ between diffeological spaces is called an \textbf{induction}
if it is injective and a function $p : U \to X$ is a plot if and only if $i \circ p$ is a plot.
For every injective function $i : X \to Y$ with $Y$ a diffeological space, there
is a unique diffeology on $X$ making $i$ an induction.
This is called the \textbf{subset diffeology}.
A smooth map $f : X \to Y$ between diffeological spaces is called a \textbf{subduction}
if every plot $p : U \to Y$ locally factors smoothly through $f$.
For every surjective function $f : X \to Y$ with $X$ a diffeological space,
there is a unique diffeology on $Y$ making $f$ a subduction.
This is called the \textbf{quotient diffeology}.

Fix a set $X$. The intersection of any family of diffeologies on $X$ is again a diffeology.
A collection of maps from open subsets of Euclidean spaces to
a fixed diffeological space \textbf{generates} the diffeology if the diffeology is the smallest diffeology
containing this collection of maps.
The \textbf{$D$-dimension} of a diffeological space $X$ is the infimum,
over all choices of sets of plots that generate the diffeology on $X$,
of the supremum of the dimension of the domains of the generating plots.
This generalizes the concept of dimension from smooth manifolds to diffeological spaces.

Given a diffeological space $X$, we have a category $\Plot(X)$, called the \textbf{plot category} of $X$.
The objects are the plots $p:U \to X$ and the morphisms are the commutative triangles
\begin{equation}\label{eq:plots}
 \vcenter{\xymatrix@C5pt{U \ar[rr]^f \ar[dr]_-p && V \ar[dl]^-q \\ & X}}
\end{equation}
with $p,q$ plots of $X$ and $f$ a smooth map. This category and the above commutative triangle
will be used throughout the paper.

Similarly, given a pointed diffeological space $(X,x)$, we have a category $\Plot(X,x)$, called the
\textbf{pointed plot category} of $(X,x)$. The objects are the pointed plots $p:(U,0) \to (X,x)$ and the
morphisms are the commutative triangles
\begin{equation}\label{eq:pointed-plots}
 \vcenter{\xymatrix@C5pt{(U,0) \ar[rr]^f \ar[dr]_p && (V,0) \ar[dl]^q \\ & (X,x)}}
\end{equation}
with $p,q$ pointed plots of $X$ and $f$ a pointed smooth map.

A subset of a diffeological space is said to be
\textbf{$D$-open} if the preimage under every plot is open in the domain,
where the domain is equipped with the usual Euclidean topology.
The $D$-open sets form a topology called the \textbf{$D$-topology}
on the underlying set of any diffeological space.
Every smooth map between diffeological spaces is continuous when the domain
and codomain are equipped with the $D$-topology.
The $D$-topology on every smooth manifold is the usual topology.

A \textbf{diffeological vector space} is both a diffeological space and a vector space such that the addition and
the scalar multiplication maps are smooth. Given two diffeological vector spaces $V$ and $W$, we write
$L^\infty(V,W)$ for the set of all smooth linear maps $V \to W$. For any vector space, there is a smallest
diffeology making it into a diffeological vector space.
This is called the \textbf{fine diffeology}. See~\cite[Chapter~3]{I2},
\cite[Section~5]{W} and~\cite[Section~3.1]{CW4} for more on fine diffeological vector spaces.

Given a diffeological space $X$, write $F(X)$ for the free vector space on
the elements of $X$.  We write $[x]$ for the basis vector associated to $x \in X$.
There is a smallest vector space diffeology on $F(X)$
such that the canonical map $X \to F(X)$ sending $x$ to $[x]$ is smooth.
$F(X)$ with this diffeology is
called the \textbf{free diffeological vector space generated by $X$}, and it can be used to define the
diffeology on the tensor product of diffeological vector spaces. It has the universal property that every smooth
map from $X$ to a diffeological vector space $V$ extends uniquely to a smooth linear map from $F(X)$ to $V$. For more
on diffeological vector spaces, see~\cite{W,CW4}.

\subsection{Tangent spaces and tangent bundles}

Tangent spaces and tangent bundles of diffeological spaces were defined in~\cite{H}
and then studied further in~\cite{CW2,CW3}. We review the basic concepts in this section.

Given a pointed diffeological space $(X,x)$, the \textbf{tangent space} $T_x(X)$
is defined as follows.  As a vector space, it is the colimit of
the functor $\Plot(X,x) \to \Vect$ sending the commutative triangle~\eqref{eq:pointed-plots} in
Section~\ref{subsection-diff} to $f_*:T_0(U) \to T_0(V)$. We will give $T_x(X)$ the subset diffeology of the
diffeology on the tangent bundle $T^{\dvs}X$ described below.

The \textbf{tangent bundle} $TX$ of a diffeological space $X$ is defined as follows.
As a set, it is
$\coprod_{x \in X} T_x(X)$. Every plot $U \to X$ induces a canonical map $TU \to TX$. The diffeology on $TX$
generated by all such maps is called \textbf{Hector's diffeology}, and this diffeological space is denoted by $T^HX$.
It was shown in~\cite[Example~4.3]{CW2} that in general the fibrewise addition and fibrewise scalar multiplication
maps for $T^HX$ are not smooth,
i.e., that $T^HX \to X$ is not a diffeological vector pseudo-bundle in the sense
of Definition~\ref{de:dvpb}.
One corrects this by enlarging the diffeology using Proposition~\ref{prop:dvsification} below,
producing the \textbf{(internal) tangent bundle} $T^{\dvs}X \to X$ of $X$.
See~\cite[Section~4.2]{CW2} for the
conceptual differences between the two diffeologies. When $T^HX = T^{\dvs}X$, we simply write $TX$. This is
the case when $X$ is a filtered diffeological space (defined in the next paragraph, \cite[Theorem~4.12]{CW3}) or when $X$ is a homogeneous space
(\cite[Proposition~4.13]{CW3}).

When working with tangent spaces, it is more convenient to work with the \textbf{germ category} $\mathcal{G}(X,x)$
for a pointed diffeological space $(X,x)$. Its objects are pointed plots $p:(U,0) \to (X,x)$, and
its morphisms are the commutative triangles as displayed in~\eqref{eq:pointed-plots}, with $f$ a
pointed smooth germ instead of a genuine smooth map. A diffeological space $X$ is called \textbf{weakly filtered} if
each germ category $\mathcal{G}(X,x)$ is weakly filtered, i.e., every finite collection of
pointed plots factor, as germs, through a common pointed plot.
A diffeological space $X$ is called \textbf{filtered} if each germ category $\mathcal{G}(X,x)$ is filtered,
i.e., if it is weakly filtered and for any two parallel morphisms there is a morphism that coequalizes them.  This is
equivalent to requiring that every finite diagram in $\mathcal{G}(X,x)$ has a cocone.
Every diffeological group (or diffeological vector space) is weakly filtered, but not all are filtered;
see~\cite[Remark~4.20]{CW3}.

\subsection{Cotangent bundles and differential forms}\label{ss:cotangent}

Fix $k \in \N$. Given a diffeological space $X$, we have a functor $\Plot(X)^{\op} \to \Vect$ sending the commutative
triangle~\eqref{eq:plots} in Section~\ref{subsection-diff} to $\Omega^k (V) \to \Omega^k (U)$. The vector space $\Omega^k(X)$
of differential $k$-forms on $X$ is defined to be the limit of this functor.

A diffeology on $\Omega^k(X)$ is defined in~\cite[6.29]{I2}, and
we give an equivalent description here.
Note that when $U$ is an open subset of some $\R^n$, $\Omega^k(U)$ has a natural
diffeology by regarding it as the space of sections of the bundle $(\Exterior^k TU)^*$.
(Since this bundle is trivial, one can equivalently think of this as the space
$C^{\infty}(U,\, \Exterior^k(\R^n))$, as is done in~\cite{I2}.)
In~\cite{I2}, a function $V \to \Omega^k(X)$ is defined to be a plot if and only if
for each plot $p : U \to X$, the component $V \to \Omega^k(U)$ is smooth.
In other words, $\Omega^k(X)$ is defined to be $\lim \Omega^k(U)$ as
diffeological vector spaces.

As a convention, given $\omega \in \Omega^k(X)$, a plot $p:U \to X$ and $u \in U$, we write $\omega_p$ for the corresponding
form on $\Omega^k(U)$, whose value at $u$ is denoted $\omega_p(u)$.

A differential $k$-form $\omega \in \Omega^k(X)$ is said to vanish at $x \in X$ if for every pointed plot $p:(U,0) \to (X,x)$, we have
$\omega_p(0) = 0$. The space of such forms is a linear subspace of $\Omega^k(X)$, and the corresponding linear quotient is denoted by
$\Uplambda^k_x(X)$; see~\cite[6.40]{I2}. Write $\Uplambda^k X$ for $\coprod_{x \in X} \Uplambda^k_x(X)$.
There is a canonical surjective map $X \times \Omega^k(X) \to \Uplambda^k X$, and we equip the codomain with the quotient diffeology (\cite[6.45]{I2}).
By Section~\ref{sss:quotients},
the canonical map $\Uplambda^k X \to X$ is a diffeological vector pseudo-bundle in the sense of
Definition~\ref{de:dvpb}.
Moreover, $\Omega^k(X)$ can be identified with the space of sections of this bundle as diffeological vector spaces, 
by using a canonical form (called the Liouville form) obtained from a canonical form on $X \times \Omega^k(X)$;
see~\cite[6.49 and~6.50]{I2}.

The following result is new, and gives a way to calculate the space of
top-dimensional forms of a diffeological space.

\begin{prop}\label{prop:top-forms}
 Let $X$ be a diffeological space of $D$-dimension $\leq n$. Write $\overline{\Omega}^n(X)$ for the limit in the category of diffeological vector spaces
 of $\Omega^n(U)$ indexed over the full subcategory of the plot category of $X$ with objects
 the plots $U \to X$ with $\dim U= n$ (called the \textbf{$n$-plots of $X$}). Then the canonical map $\Omega^n(X) \to \overline{\Omega}^n(X)$
 is an isomorphism of diffeological vector spaces.
\end{prop}

Notice that this proposition refines condition (2) of~\cite[6.41]{I2} under the $D$-dimension assumption.

\begin{proof}
 Injectivity follows from~\cite[6.37]{I2}.

 For surjectivity, we need to show that every $\omega \in \overline{\Omega}^n(X)$ defines an $n$-form on $X$.
 Fix a compatible family of $n$-forms $\omega_p \in \Omega^n(U)$ for each
 $n$-plot $p : U \to X$.
 We must extend this to a compatible family indexed by all plots.
 Given a plot $q : V \to X$, the assumption that $X$ has $D$-dimension $\leq n$
 implies that $V$ has an open cover $\{V_i\}$ such that for each $i$,
 $q|_{V_i} : V_i \to X$ factors as $V_i \xrightarrow{f_i} U_i \xrightarrow{p_i} X$,
 where $\dim U_i = n$.
 We will show that the forms $f_i^*(\omega_{p_i})$ agree on the intersections
 and therefore give a well-defined form $\omega_q$ on $V$.
 Generalizing slightly, we show that given a commutative square
 \[
  \xymatrix{V \ar[r]^{f_1} \ar[d]_{f_2} & U_1 \ar[d]^{p_1} \\ U_2 \ar[r]_{p_2} & X}
 \]
 with $p_1$ and $p_2$ $n$-plots of $X$ and $V$ any open subset of Euclidean space,
 we have $f_1^*(\omega_{p_1}) = f_2^*(\omega_{p_2})$.
 By~\cite[6.37]{I2}, this is true if it is true after pullback along any
 smooth map $U \to V$, where $\dim U = n$, and this follows from the assumption
 that $\omega$ is compatible with the $n$-plots of $X$.
 The same reasoning shows that the new family extends the original family
 and is compatible with all plots of $X$.

 It is straightforward to see that the canonical map $\Omega^n(X) \to \overline{\Omega}^n(X)$ is smooth.
 To see that the inverse map is also smooth, we start with a plot $W \to \overline{\Omega}^n(X)$
 and show that it smoothly factors through the canonical map.
 So, for each plot $q : V \to X$, we must produce a smooth map $W \to \Omega^n(V)$.
 As in the previous paragraph,
 by assumption we have that there exists an open cover $\{V_i\}$ of $V$ such that, for each $i$,
 $q|_{V_i}$ factors as $V_i \xrightarrow{f_i} U_i \xrightarrow{p_i} X$, where $\dim U_i = n$.
 Consider the smooth composite $W \to \overline{\Omega}^n(X) \to \Omega^n(U_i) \to \Omega^n(V_i)$,
 for each $i$.  Again following the argument of the previous paragraph, these match up
 to give a smooth map $W \to \Omega^n(V)$, and these form a compatible family
 extending the original family defined on $n$-plots.
 That is, we get a smooth map $W \to \Omega^n(X)$ whose composite with the
 canonical map gives the original plot $W \to \overline{\Omega}^n(X)$.
\end{proof}

\begin{rem}
 Recall from~\cite[6.37]{I2} that two $k$-forms on a diffeological space $X$ are equal if they coincide
 when evaluated at every $k$-plot of $X$.
 However, if $k$ is less than the $D$-dimension of $X$,
 then a compatible family of $k$-forms on the $k$-plots of $X$ may \emph{not} give rise to a $k$-form on $X$.
 For example, a $0$-form on $\R$ is a smooth function $\R \to \R$, while a compatible family of 
 $0$-forms on the $0$-plots of $\R$ gives rise to an arbitrary function $\R \to \R$. 
 In summary, a $k$-form on a diffeological space is determined by the $k$-plots, but not defined by 
 the $k$-plots, unless $k$ is greater than or equal to the $D$-dimension of the space.
 Notice that this does not conflict with the content of~\cite[6.37]{I2}; see the last paragraph 
 before the proof there.
\end{rem}

\section{Operations on diffeological vector pseudo-bundles}\label{se:operations}

In this section, we provide a general theory of diffeological vector pseudo-bundles.
Section~\ref{ss:dvpb} gives the basic definitions and then proves
fundamental results about the categories of all diffeological vector pseudo-bundles
as well as those over a given base.
The construction of limits and colimits in Theorem~\ref{thm:DVPB-bicomplete} is
used in many places.
Section~\ref{ss:fibrewise} describes a variety of fibrewise operations
on diffeological vector pseudo-bundles over a common base, extending
work of~\cite{V} and~\cite{P}.
We will use these tools when discussing differential forms in Section~\ref{se:duality}.

\subsection{Diffeological vector pseudo-bundles}\label{ss:dvpb}

In~\cite{V}, Vincent defined ``regular vector bundle.''
While not literally the same, Vincent's notion is equivalent to the
notion of a ``diffeological vector space over $X$'',
as defined in~\cite[Definition~4.5]{CW2}.
We recall the latter definition here, but use the terminology of~\cite{P}
to avoid confusion with ordinary diffeological vector spaces.

\begin{de}\label{de:dvpb}
 Let $B$ be a diffeological space. A \textbf{diffeological vector pseudo-bundle} over $B$ is a diffeological space
 $E$, a smooth map $\pi:E \to B$ and a vector space structure on each of the fibres $E_b := \pi^{-1}(b)$ such that
 the fibrewise addition map $E \times_B E \to E$, the fibrewise scalar multiplication map $\R \times E \to E$ and the zero
 section $\sigma:B \to E$ are all smooth. We call $B$ the base space, $E$ the total space, and $E_b$ the fibre at $b$.

 A diffeological vector pseudo-bundle $\pi:E \to B$ is called \textbf{trivial} if there is a diffeomorphism
 $f:E \to B \times V$ for some diffeological vector space $V$ such that $\pi = \Pr_1 \circ f$ and the restriction $f_b:E_b \to V$
 of $f$ is linear for each $b \in B$. It is called \textbf{locally trivial} if there exists a $D$-open cover $\{B_i\}$ of
 $B$ such that the restriction of $\pi$ to $\pi^{-1}(B_i) \to B_i$ is trivial for each $i$.
\end{de}

As noted in~\cite[Remark~4.12]{CW2}, a diffeological vector pseudo-bundle $\pi:E \to B$ is not necessarily a diffeological bundle (in
the sense of~\cite{I1,I2}) or a fibration (in the sense of~\cite{CW1}). The essential difference between diffeological vector
pseudo-bundles and classical vector bundles is that we remove the local triviality condition.
We do this so that we can include tangent bundles, which are often not locally trivial.

\begin{ex}
 Let $X$ be a diffeological space. Then $\pi:T^\dvs X \to X$ is a diffeological vector pseudo-bundle, where
 $T^\dvs X$ is the tangent bundle of $X$ equipped with the dvs diffeology; see~\cite[Section~4.1]{CW2}
 for details.
\end{ex}

The following result is very useful:

\begin{prop}[{\cite[Proposition~4.6]{CW2}}]\label{prop:dvsification}
 Let $\pi:E \to B$ be a smooth map between diffeological spaces, and suppose that each fibre of $\pi$ has a vector
 space structure.
 Then there is a smallest diffeology on $E$ which contains the given diffeology and which makes $\pi$ into a
 diffeological vector pseudo-bundle over $B$.
\end{prop}

We call $\pi$ a \textbf{diffeological vector pre-bundle}, and call the
operation described by the proposition \textbf{dvsification}.
By~\cite[Section~4.2]{CW2}, it can be seen as a left adjoint to a forgetful functor.

Concretely, every plot in the new diffeology of the total space is locally of the following form. Given a plot
$q:U \to B$, plots $q_1,\ldots,q_k:U \to E$ such that $\pi \circ q_i = q$ for all $i$, and plots
$r_1,\ldots,r_k:U \to \R$, the linear combination $U \to E$ sending $u$ to $\sum_i r_i(u) q_i(u)$ in $E_{q(u)}$ is
a plot in the new diffeology. Note that when $k=0$, this is the composite of the plot $q$ of $B$ with the zero
section of $\pi$.

Note that in the previous proposition, the original diffeology on the fibre $E_b$ might differ from the new
diffeology; see~\cite[Example~4.3]{CW2} for an example. But if each $E_b$ is already a diffeological vector space
in the original diffeology, then the new diffeology coincides with the original one on the fibre.
More generally, dvsification commutes with pullbacks; see~\cite[Section~4.2]{CW2}.

\begin{de}
 Let $\pi_i:E_i \to B_i$ be diffeological vector pseudo-bundles for $i=1,2$.
 A \textbf{bundle map} from $\pi_1$ to $\pi_2$ consists of a pair of smooth
 maps $f:E_1 \to E_2$ and $g : B_1 \to B_2$ such that the diagram
 \begin{equation}\label{eq:comm-square-DVPB}
  \cxymatrix{E_1 \ar[r]^f \ar[d]_{\pi_1} & E_2 \ar[d]^{\pi_2} \\ B_1 \ar[r]_g & B_2}
 \end{equation}
 commutes and for each $b \in B_1$, the map $E_{1,b} \to E_{2,g(b)}$ induced by $f$ is linear.
 When the bundles are clear, we denote this bundle map by $f : E_1 \to E_2$.
 In the case that
 $B_1 = B_2 =: B$ and $g = 1_B$, we say that $f$ is a \textbf{bundle map over $B$}.
 Bundle maps are composed in the obvious way.
\end{de}

The collection of diffeological vector pseudo-bundles with bundle maps forms a
category denoted $\DVPB$.
We equip the set $\DVPB(\pi_1, \pi_2)$ with the subset diffeology from the
product $C^{\infty}(E_1, E_2) \times C^{\infty}(B_1, B_2)$.
It is easy to check that with this choice, $\DVPB$ is enriched over $\Diff$.
For a fixed diffeological space $B$, the subcategory $\DVPB_B$ consists of diffeological
vector pseudo-bundles over $B$ together with bundle maps over $B$.
In $\DVPB_B$, the hom-sets are vector spaces under
pointwise addition and pointwise scalar multiplication, and composition is bilinear.
In particular, $\DVPB_B$ is $\Ab$-enriched.
Moreover, when $\DVPB_B(\pi_1,\pi_2)$ is equipped with the subset diffeology from
$C^{\infty}(E_1, E_2)$ (or, equivalently, from $\DVPB(\pi_1, \pi_2)$), it is a diffeological vector space,
and one can check that $\DVPB_B$ is enriched over the category of diffeological vector spaces.

Given a diffeological vector pseudo-bundle $\pi:E \to B$, a \textbf{section} is a smooth map $s:B \to E$
such that $\pi \circ s = 1_B$. The set $\Gamma(\pi)$ of all sections of $\pi$ is a vector space under
pointwise addition and pointwise scalar multiplication. We equip $\Gamma(\pi)$ with
the subset diffeology of $C^\infty(B,E)$. This makes $\Gamma(\pi)$ into a diffeological vector space.
Moreover, $\Gamma$ is a functor from
$\DVPB_B$ to the category of diffeological vector spaces.
For a trivial bundle $\Pr_1:B \times V \to B$, we have $\Gamma(\Pr_1) = C^\infty(B,V)$.

Let $\pi: E \to B$ be a diffeological vector pseudo-bundle,
and let $f: B' \to B$ be a smooth map.
Then the pullback $f^*(\pi) : f^*(E) \to B'$ of $\pi$ in $\Diff$
is also a diffeological vector pseudo-bundle.
Here $f^*(E) := \{(b',e) \in B' \times E \mid f(b')=\pi(e)\}$
and $f^*(\pi)$ is the projection map.
The fibre of $f^*(\pi)$ over $b'$ is isomorphic to $E_{f(b')}$
as diffeological vector spaces for each $b' \in B'$.
Note that pullback gives a functor $f^* : \DVPB_B \to \DVPB_{B'}$.
As we will do for many functors, we write $f^*(\pi)$ for the bundle and
$f^*(E)$ for the total space, even though the latter of course
depends on $\pi$ as well.

The pullback of a trivial bundle is trivial,
and the pullback of a locally trivial bundle is locally trivial.

Given a pullback diagram
\[
 \xymatrix{E' \ar[r] \ar[d]_{\pi'} & E \ar[d]^\pi \\ B' \ar[r]_f & B}
\]
with $\pi'$ and $\pi$ diffeological vector pseudo-bundles, we have a smooth linear map
$\Gamma(\pi) \to \Gamma(\pi')$ sending $s:B \to E$ to $B' \to E'=f^*(E)$
defined by $b' \mapsto (b',s(f(b')))$. When $f$ is surjective, this map is injective.

The categories $\DVPB$ and $\DVPB_B$ have good properties.

\begin{thm}\label{thm:DVPB-bicomplete}
 The category $\DVPB$ is complete and cocomplete.
 For each diffeological space $B$, the category $\DVPB_B$ is complete, cocomplete
 and additive.
\end{thm}

\begin{proof}
 The first claim was proved in~\cite[Section~4.2]{CW2}.
 We sketch the idea of the proof since we will use a similar approach
 for the second claim and will also make use of the explicit construction later.
 Given a functor $F:\mathcal{I} \to \DVPB$ from a small category, write $tF$ and $bF$ for the functors
 sending $i \in \Obj(\mathcal{I})$ to the total space and the base space of $F(i)$, respectively.

 The limit of $F$ is easy to construct: it is given by the canonical map
 $\lim tF \to \lim bF$, which is a diffeological vector pseudo-bundle.

 The colimit of $F$ is rather complicated. The colimit $B$ of $bF$ serves as the base space of the colimit of $F$.
 The total space of the colimit is constructed as follows.
 Let $\cC$ be the category of elements of $bF$,
 whose objects are pairs $(i,a)$ with $i \in \Obj(\mathcal{I})$ and $a \in bF(i)$, and whose morphisms $(i,a) \to (i',a')$
 are the morphisms $i \to i'$ in $\mathcal{I}$ such that the induced map
 $bF(i) \to bF(i')$ sends $a$ to $a'$.
 For each $b \in B$, consider the full subcategory $\cC_b$ of $\cC$
 consisting of the objects $(i,a)$ such that $a \in bF(i)$ is sent to $b$
 by the cocone map $bF(i) \to B$.
 There is a functor $\cC_b \to \Vect$ sending $(i,a)$ to the fibre $tF(i)_a$.
 Define the fibre $E_b$ over $b$ to be the colimit of this functor.
 Let $E := \coprod_b E_b$. There is a natural map $E \to B$.
 Equip $E$ with the dvsification of the smallest diffeology such that all of the canonical maps from $tF(i) \to bF(i)$ to $E \to B$
 are smooth. Then $E \to B$ is the desired colimit of $F$.

 For the second claim, let $F:\cI \to \DVPB_B$ be a functor
 from a small category, and again write $tF$ for the functor sending
 $i \in \Obj(\mathcal{I})$ to the total space of $F(i)$.

 Let $E' \to B'$ be the limit of $F$ in the category $\DVPB$.
 There is a natural map $k : B \to B' = \lim_i B$ induced by the cone
 sending each $i$ to $1_B$.
 One can check directly that taking $E \to B$ to be the pullback of
 $E' \to B'$ along $k$ gives a limit of $F$ in $\DVPB_B$.

 Constructing the colimit of $F$ is more involved.
 For each $b \in B$, consider the functor $\cI \to \Vect$ sending
 $i$ to $tF(i)_b$, the fibre of $F(i)$ above $b$.
 Define $E_b$ to be the colimit of this functor, and let $E := \coprod_b E_b$,
 which comes with a natural map $E \to B$.
 Equip $E$ with the dvsification of the smallest diffeology such that
 all of the canonical maps $tF(i) \to E$ (over $B$) are smooth.
 To see that this is the colimit of $F$, let $C \to B$ be in $\DVPB_B$
 and consider a cocone consisting of maps $tF(i) \to C$ over $B$ for each $i$,
 commuting with the maps in the diagram.
 For each $b \in B$, we get a cocone of linear maps $tF(i)_b \to C_b$
 which therefore induce a linear map $E_b := \colim_i tF(i)_b \to C_b$.
 Taking together, these give a function $E \to C$ over $B$, unique
 subject to the condition that it is a (set-theoretic) map of cocones.
 Since it is a map of cocones, it is smooth with respect to the diffeology
 generated by the canonical maps $tF(i) \to E$.
 And since $C \to B$ is a diffeological vector pseudo-bundle, it remains
 smooth when $E$ is equipped with the dvsification of this diffeology.
 Therefore $E \to B$ is the desired colimit of $F$.

 We mentioned above that $\DVPB_B$ is $\Ab$-enriched, and since it
 has finite products, it is additive.
\end{proof}

\begin{rem}\label{rem:connected-limit}
 Let $F : \cI \to \DVPB_B$ be a functor with $\cI$ connected.
 Then the natural map $B \to B' := \lim_i B$ is a diffeomorphism,
 so the limit in $\DVPB_B$ is also the limit in $\DVPB$.
 This is illustrated with the discussion of kernels in
 Section~\ref{sse:subbundle-kernel}.
\end{rem}

\begin{rem}
 While the category $\DVPB_B$ is additive, it is not abelian.
 We'll see below that it has kernels and cokernels, but
 in general not every monic is a kernel, nor is every epi a cokernel.
 Taking $B$ to be a point, this specializes to the fact that the
 category of diffeological vector spaces and smooth linear maps is
 additive but not abelian, which was observed in~\cite[Remark~3.18]{W}.
\end{rem}

We next observe that colimits in $\DVPB$ are enriched colimits.

\begin{prop}\label{prop:colimit-enriched}
 Let $F:\cI \to \DVPB$ be a functor and let $\pi': E' \to B'$ be a diffeological
 vector pseudo-bundle. Then the natural bijection
\[
  \DVPB(\colim_{i \in \cI} F, \pi') \cong \lim_{i \in \cI} \DVPB(F(i), \pi') ,
\]
 is a diffeomorphism.
\end{prop}

\begin{proof}
Write $\pi : E \to B$ for the colimit of $F$.
This is equipped with a cone map $F(i) \to \pi$ for each $i \in I$.
Precomposing with the cone map gives a smooth map
$\DVPB(\pi, \pi') \to \DVPB(F(i), \pi')$, since $\DVPB$ is enriched over $\Diff$.
These maps induce the natural bijection in the statement, showing that
this bijection is smooth.

To see that the inverse is smooth, recall that $\DVPB(\pi, \pi')$ is equipped with the
subset diffeology from $C^{\infty}(E, E') \times C^{\infty}(B, B')$, so we need
to show that the natural map
\begin{equation*}\label{eq:natmap}
  \lim_{i \in \cI} \DVPB(F(i), \pi') \lra C^{\infty}(E, E') \times C^{\infty}(B, B')
\end{equation*}
is smooth.  For this, we need to show that each component is smooth.
The second component is smooth because it is the composite of two smooth maps
\[
  \lim_{i \in \cI} \DVPB(F(i), \pi') \lra \lim_{i \in \cI} C^\infty(bF(i), B') \cong C^{\infty}(B, B') ,
\]
where $F(i) : tF(i) \to bF(i)$ and we are using that $B = \colim_{i \in \cI} bF(i)$.
To show that the first component is smooth, we need to show that its adjoint
\[
  \gamma : \lim_{i \in \cI} \DVPB(F(i), \pi') \times E \lra E'
\]
is smooth.
A plot of the limit consists of compatible plots $p_i : U \to C^{\infty}(tF(i), E')$
and $q_i : U \to C^{\infty}(bF(i), B')$ for $i \in I$.
To start, consider a plot of $E$ of the form $U \xrightarrow{r} tF(j) \to E$,
where $tF(j) \to E$ is part of the colimiting cone.
These plots together give a plot of the domain of $\gamma$, which when
composed with $\gamma$ give the map $U \to E$ sending $u$ to $p_j(u)(r(u))$.
This is smooth because the evaluation map $C^{\infty}(tF(j), E') \times tF(j) \to E'$ is smooth.
As a result of the dvsification performed in the construction of $\pi$,
a general plot in $E$ is locally a fibrewise sum of plots like the one
we considered.
Since $\pi' : E' \to B'$ was assumed to be a diffeological vector pseudo-bundle,
it follows that the composite of a general plot with $\gamma$ is smooth.
\end{proof}

Similar statements hold for limits in $\DVPB$ and limits and colimits in $\DVPB_B$.

We will later need the following result which identifies maps in $\DVPB_B$ out of
a colimit taken in $\DVPB$.
For a functor $F : \cI \to \DVPB$, we continue to write $bF(i)$ for the base of $F(i)$.

\begin{prop}\label{prop:colimit-limit}
 Let $F:\cI \to \DVPB$ be a functor, and write $B := \colim \, bF$ for the colimit in $\Diff$. Let $\pi:E \to B$ be a diffeological
 vector pseudo-bundle. Then there is a natural diffeomorphism
 \[
  \DVPB_B(\colim_{i \in \cI} F,\pi) \cong \{y \in \lim_{i \in \cI} \DVPB(F(i),\pi) \mid y_i \text{ lies over } bF(i) \to B\} ,
 \]
 where the diffeology on the right-hand-side is the subset diffeology of the limit diffeology.
\end{prop}

\begin{proof}
 By Proposition~\ref{prop:colimit-enriched}, we have a diffeomorphism
\[
  \DVPB(\colim_{i \in \cI} F, \pi) \cong \lim_{i \in \cI} \DVPB(F(i), \pi) .
\]
 The diffeological spaces in the statement are subspaces of these, so we simply
 need to check that the bijection restricts to the given subsets.
 The map from left to right sends a bundle map $f : \colim F \to \pi$ over $B$
 to the element $y \in \lim \DVPB(F(i), \pi)$ whose component $y_i$ is
 $f$ composed with the cone map $F(i) \to \colim F$.
 Thus $y_i$ is a map over the base $bF(i) \to B$ of this cone map.
 Conversely, any such element $y$ will induce a map $\colim F \to \pi$
 over the identity map of $B$.
\end{proof}

\subsection{Fibrewise operations}\label{ss:fibrewise}

In this section, we discuss fibrewise operations on diffeological vector
pseudo-bundles over a fixed base space $B$.
We begin with some operations that are simply limits and colimits in $\DVPB_B$,
and then move on to other operations, including the exterior power
and dual operations which will be used in Section~\ref{se:duality}.
Some of these definitions also appear in~\cite{V} and~\cite{P}.
However,~\cite{V} does not discuss the tensor product of diffeological
vector pseudo-bundles (called ``regular vector bundles'' there).
These are discussed in~\cite{P}, but the definition given there is
incorrect.
Neither reference discusses free bundles, $\Hom$ bundles, or the
subbundle reachable by sections.

\subsubsection{Products}

Let $\{\pi_i:E_i \to B\}_{i \in I}$ be a family of diffeological vector pseudo-bundles over the same base space $B$.
We now describe the product of these bundles in $\DVPB_B$, following the
proof of Theorem~\ref{thm:DVPB-bicomplete}.
First, the product in $\DVPB$ is simply $\prod_i \pi_i:\prod_i E_i \to \prod_i B$.
Then the \textbf{(fibrewise) product} $\prod_{i,B} \pi_i : \prod_{i,B} E_i \to B$ is
the pullback of $\prod_i \pi_i$ along the diagonal map
$\Delta:B \to \prod_i B$.
The fibre of $\prod_{i,B} \pi_i$ over $b$ is isomorphic to $\prod_i E_{i,b}$ for each $b \in B$.
Moreover, $\Gamma(\prod_{i,B} \pi_i) \cong \prod_i \Gamma(\pi_i)$ as diffeological vector spaces.

The product of trivial bundles is again trivial, or more precisely
$\prod_{i,B} (B \times V_i) \cong B \times (\prod_i V_i)$.
Moreover, given a smooth map $f : B' \to B$, the pullback operation
$f^* : \DVPB_B \to \DVPB_{B'}$ commutes with products in the sense
that there is a natural isomorphism
$f^*(\prod_{i,B} \pi_i) \cong \prod_{i,B'}(f^*(\pi_i))$.
This can be checked directly.
(A conceptual point of view is that the total space on both sides
of the isomorphism can be seen as the limit, in $\Diff$, of the
diagram consisting of the map $f$ and all of the maps $\pi_i$.)
It follows from these two facts that a finite product of locally trivial bundles is again locally trivial.

\subsubsection{Direct sums}

Let $\{\pi_j:E_j \to B\}_{j \in J}$ be a family of diffeological vector pseudo-bundles over the same base space $B$.
The \textbf{direct sum} $\bigoplus_{j \in J} \pi_j: \bigoplus_{j \in J} E_j \to B$ is defined to
be the coproduct in $\DVPB_B$.
Following the proof of Theorem~\ref{thm:DVPB-bicomplete},
the total space is $\coprod_{b \in B} (\bigoplus_{j \in J} E_{j,b})$, as a set.
For each $j' \in J$, there is a natural map $E_{j'} \to \bigoplus_{j \in J} E_j$.
We define the diffeology on $\bigoplus_{j \in J} E_j$ to be the
dvsification of the diffeology generated by these natural maps.
It follows that the subset diffeology on each fibre $\bigoplus_{j \in J} E_{j,b}$
is the dvsification of the diffeology generated by the natural inclusions
$E_{j',b} \to \bigoplus_{j \in J} E_{j,b}$,
which implies that $\bigoplus_{j \in J} E_{j,b}$ is a coproduct in the category of
diffeological vector spaces.
More generally, the pullback along a smooth map $B' \to B$ commutes with direct sums.

There is a natural smooth injective bundle map
$\bigoplus_{j \in J} E_j \to \prod_{j \in J, B} E_j$.
Since the category $\DVPB_B$ is $\Ab$-enriched, this map is an isomorphism
when $J$ is finite.
For general $J$, this injective map is not necessarily an induction;
see~\cite[Example~5.4]{W} for such an example.

\begin{rem}
 There is a natural injective map $\bigoplus_{j \in J} \Gamma(\pi_j) \to \Gamma(\bigoplus_j \pi_j)$,
 which is an isomorphism when $J$ is finite.
 In general, it can be a proper inclusion.
 For example, let $\pi$ be the countable direct sum of copies of $\Pr_1:\R^2 \to \R$,
 and let $f_j:\R \to \R$ be a smooth map such that $\Supp(f_j) =[j,j+1]$ for each $j \in \N$.
 Then $\pi$ is $\Pr_1: E := \R \times (\bigoplus_\omega \R) \to \R$, and $\R \to E$ defined by
 $x \mapsto (x,(f_j(x)))$ is a section of $\pi$, but it is not in $\bigoplus_\omega \Gamma(\Pr_1)$.
\end{rem}

\subsubsection{Subbundles and kernels}\label{sse:subbundle-kernel}

Let $\pi:E \to B$ be a diffeological vector pseudo-bundle. We say that a diffeological vector pseudo-bundle
$\pi':E' \to B$ is a \textbf{subbundle} of $\pi$ if there is a bundle map $i:E' \to E$ over $B$ which is an induction. In
this case, the restriction $i_b:E'_b \to E_b$ is also an induction for each $b \in B$.
For example, if $E'$ is a subset of $E$ equipped with the subset diffeology
such that $E'_b$ is a subspace of $E_b$ for each $b \in B$,
then the restriction of $\pi$ to $E'$ gives a subbundle $E' \to B$ of $\pi$.

Here is a source of examples.
Let $f:E' \to E$ be a bundle map over $B$ between two diffeological vector
pseudo-bundles $\pi':E' \to B$ and $\pi:E \to B$.
Then the \textbf{kernel} of $f$ is the subbundle
$K := \{ e \in E' \mid f(e) = 0 \}$ of $E'$.
This is the equalizer of $f$ and the zero map in $\DVPB_B$,
and also in $\DVPB$ (see Remark~\ref{rem:connected-limit}).
One can check that pullback along a smooth map $B' \to B$ commutes with kernels.

\subsubsection{Quotients and cokernels}\label{sss:quotients}

Dually, a diffeological vector pseudo-bundle $\pi':E' \to B$ is called a \textbf{quotient} of another
diffeological vector pseudo-bundle $\pi:E \to B$ if there is a bundle map $q:E \to E'$ over $B$ which is a
subduction. In this case, the restriction $q_b:E_b \to E_b'$ is also a
subduction for each $b \in B$.
For example, if $E'' \subseteq E$ is a subbundle, write $E/E''$ for the quotient
of $E$ by the equivalence relation $\sim$ with $e_1 \sim e_2$ if $e_1$ and $e_2$
are in the same fibre and $e_1 - e_2 \in E''$.
Give $E/E''$ the quotient diffeology.  Then $\pi$ induces a natural smooth map
$E/E'' \to B$.
This is already a diffeological vector pseudo-bundle,
without having to apply dvsification, since plots in $E/E''$ locally lift
to $E$ where they can be added and multiplied by scalars.

As a special case, the \textbf{cokernel} of a map $f : E'' \to E$ in $\DVPB_B$
is $E/f(E'') \to B$.
This is the coequalizer of $f$ and the zero map in $\DVPB_B$.
One can check that pullback along a smooth map $B' \to B$ commutes with cokernels.

Let $E \to E'$ be a quotient bundle map over $B$, and let $E''$ be its kernel.
Then $E'$ is isomorphic to the cokernel $E/E''$ as diffeological vector pseudo-bundles over $B$.

\subsubsection{Free bundles}
In order to define tensor products, we introduce a fibrewise free vector space
construction.
Let $f:X \to B$ be a smooth map between diffeological spaces. As a set, let $F_B(X)$ be the disjoint union
$\coprod_{b \in B} F(X_b)$, where $F(X_b)$ denotes the free vector space generated by the
diffeological subspace $X_b := f^{-1}(b)$ of $X$ for each $b \in B$.
There is a canonical map $F_B(f) : F_B(X) \to B$ sending $v \in F(X_b)$ to $b$. Equip $F_B(X)$ with the
dvsification of the diffeology generated by the canonical map $i_X:X \to F_B(X)$ sending $x \in X_b$ to
$[x] \in F(X_b)$.
We give an explicit description of this diffeology:

\begin{prop}\label{pr:free}
 Given plots $q:U \to B$ and
 $\{q_i:U \to X\}_{i=1}^n$ such that $q = f \circ q_i$ for each $i$, together with smooth maps
 $\{r_i:U \to \R\}_{i=1}^n$, the map $U \to F_B(X)$ defined by $u \mapsto \sum_{i=1}^n r_i(u) [q_i(u)]$ is a plot
 of $F_B(X)$,
 and every plot of $F_B(X)$ is locally of this form.
 Note that when $n=0$, this is the composite of the plot $q$ of $B$ with the zero section
 $B \to F_B(X)$.
\end{prop}
\begin{proof}
 This follows directly from the explicit description of the dvsification; see the second paragraph after
 Proposition~\ref{prop:dvsification}.
\end{proof}

As an easy consequence of this proposition, the subset diffeology on the fibre of $F_B(X)$ over $b \in B$ is exactly
the diffeology of the free diffeological vector space generated by $X_b$; see~\cite[Proposition~3.5]{W} for more
details about the free diffeological vector space generated by a diffeological space.

The universal property of $F_B(f)$ is similar to the universal property of the free diffeological vector space
generated by a diffeological space; see the proof of~\cite[Proposition~3.5]{W} or the end of
Section~\ref{subsection-diff}. More precisely, given any commutative triangle in $\Diff$:
\[
 \xymatrix@C5pt{X \ar[rr]^g \ar[dr]_-f && E \ar[dl]^-{\pi} \\ & B,}
\]
where $\pi$ is a diffeological vector pseudo-bundle, there exists a unique bundle map $\tilde{g}:F_B(X) \to E$
over $B$ such that $g = \tilde{g} \circ i_X$. In other words, the free bundle functor
is left adjoint to the forgetful functor
$\DVPB_B \to \Diff/B$.
Moreover, if we write $\Mor(\Diff)$ for the category with objects the morphisms in $\Diff$
and morphisms the commutative squares, with the obvious composition, then the free
bundle construction gives a functor $\Mor(\Diff) \to \DVPB$
which is the left adjoint to the forgetful functor $\DVPB \to \Mor(\Diff)$.

If $\pi : B \times Y \to B$ is the projection, then
$F_B(\pi)$ is isomorphic to the projection $B \times F(Y) \to B$
as diffeological vector pseudo-bundles over $B$.
Moreover, one can check that pullback along a smooth map $B' \to B$
commutes with the free bundle construction.
It follows that if a smooth map $f : X \to B$ is locally trivial,
then so is the diffeological vector pseudo-bundle $F_B(f) : F_B(X) \to B$.

\subsubsection{Tensor products}
Given diffeological vector pseudo-bundles $\pi_i:E_i \to B$ for $i=1,2$, we define the \textbf{tensor
product} $\pi_1 \otimes \pi_2 : E_1 \otimes E_2 \to B$ as follows. First, we form the product (which is the same
as the direct sum) $E_1 \times_B E_2 \to B$, with each fibre isomorphic to $E_{1,b} \times E_{2,b}$.
Then we use the free functor to get a diffeological vector pseudo-bundle $F_B(E_1 \times_B E_2) \to B$
so that each fibre is the free diffeological vector space generated by $E_{1,b} \times E_{2,b}$. Finally, we form the
quotient diffeological vector pseudo-bundle so that as vector spaces each fibre is $E_{1,b} \otimes E_{2,b}$.
Since these constructions respect pullbacks, each fibre has the tensor product diffeology (see~\cite[Section~2.3]{V} and~\cite[Section~3]{W}).

There is a canonical smooth map
$\psi : E_1 \times_B E_2 \to E_1 \otimes E_2$ over $B$ which is fibrewise bilinear,
and this is the initial fibrewise bilinear map to a diffeological vector
pseudo-bundle over $B$:
for any diffeological vector pseudo-bundle $E \to B$,
composition with $\psi$ gives a bijection between bundle maps $E_1 \otimes E_2 \to E$ over $B$ and
smooth maps $E_1 \times_B E_2 \to E$ over $B$ which are fibrewise bilinear.
By Proposition~\ref{pr:free}, we can concretely describe the plots
$V \to E_1 \otimes E_2$ as the maps that can be \emph{locally} expressed as a finite sum
$\sum_i r_i(v) \, q_{i,1}(v) \otimes q_{i,2}(v)$,
where $r_i : V \to \R$ and $q_{i,j} : V \to E_j$ are smooth,
and $\pi_1 \circ q_{i,1} = \pi_2 \circ q_{i,2}$ for all $i$.
In other words, the diffeology on $E_1 \otimes E_2$ can be described as the
dvsification of the diffeology generated by the map $\psi$,
matching the description given in~\cite[Definition~5.2.1]{V}.

Because the product operation is commutative and associative, so is the tensor product.
The tensor product of two trivial bundles is again trivial:
$(B \times V) \otimes (B \times V') \cong B \times (V \otimes V')$.
Moreover, pullback along a smooth map $B' \to B$ commutes with tensor products.
It follows that the tensor product of locally trivial bundles is again locally trivial.

Given a diffeological vector pseudo-bundle $\pi:E \to B$, we can form the tensor bundle $T_B(\pi):T_B(E) \to B$ with
$T_B(\pi) := \bigoplus_{i \in \N} \pi^{\otimes i}$, where $\pi^{\otimes 0}$ is understood to be the trivial bundle
$B \times \R \to B$, $\pi^{\otimes 1}$ is the original bundle $\pi:E \to B$,
and $\pi^{\otimes i}:E^{\otimes i} \to B$ for $i \geq 2$. It is straightforward to check that
the fibre of $T_B(\pi)$ at $b$ is exactly $\bigoplus_{i \in \N} E_b^{\otimes i}$, and that the fibrewise tensor
multiplication map is smooth.

\subsubsection{Exterior powers}
Given a diffeological vector pseudo-bundle $\pi : E \to B$ and an integer $k$ with $k \geq 2$, we can form the $k$-fold
tensor bundle $\pi^{\otimes k} : E^{\otimes k} \to B$. Taking the usual fibrewise linear quotient of this tensor bundle equipped
with the quotient diffeology, we get the $k$-fold exterior bundle $\Exterior^k \pi : \Exterior^k E \to B$ whose fibre over $b \in B$
is $\Exterior^k E_b$ with the quotient diffeology of $E_b^{\otimes k}$.
Note that fibrewise antisymmetrization defines a smooth map $a: E^{\otimes k} \to E^{\otimes k}$
whose image can be identified with $\Exterior^k E$.
Since $a$ is idempotent, it is clear that the quotient diffeology on $\Exterior^k E$
agrees with the subset diffeology coming from $E^{\otimes k}$,
as observed in~\cite[Lemma~2.11]{P2}.
As was done for the tensor product, the plots can be explicitly described as those
that can be locally expressed as finite linear combinations of wedge products
of compatible plots in $E$;
this is made explicit in the proof of Theorem~\ref{thm:recover-forms}.

We also define $\Exterior^1 \pi$ to be $\pi$ and $\Exterior^0 \pi$ to be
the trivial bundle $B \times \R \to B$.
We can form the exterior bundle $\Exterior \pi := \bigoplus_{k \in \N} (\Exterior^k \pi)$,
which is a map $\Exterior E := \bigoplus_{k \in \N} (\Exterior^k E) \to B$.
It is straightforward to check that the fibrewise algebraic quotient map $T_B(E) \to \Exterior E$
is a quotient bundle map over $B$, so that the exterior multiplication map on $\Exterior E$ is smooth.

One can check directly that $\Exterior^k (B \times V) \cong B \times \Exterior^k V$ as
diffeological vector pseudo-bundles over $B$.
Moreover, pullback along a smooth map $B' \to B$ commutes with exterior powers.
It follows that the exterior powers of a locally trivial bundle are again locally trivial.

For each $k \in \N$, both the $k$-fold tensor power and the $k$-fold
exterior power give functors $\DVPB \to \DVPB$ in an obvious way.

While we don't make use of it in this paper, one can similarly define
the $k$-fold symmetric powers $\textup{Sym}^k(\pi)$.

\subsubsection{Hom and dual bundles}\label{sss:Hom}
Given diffeological vector pseudo-bundles $\pi_1 : E_1 \to B$ and $\pi_2 : E_2 \to B$, we
define a new diffeological vector pseudo-bundle $\Hom_B(\pi_1, \pi_2) : \Hom_B(E_1,E_2) \to B$ as follows.
As a set, the total space is the disjoint union $\coprod_{b \in B} L^\infty(E_{1,b},E_{2,b})$ of smooth linear maps
$E_{1,b} \to E_{2,b}$, where $E_{1,b}$ and $E_{2,b}$ are equipped with the subset diffeologies from $E_1$ and $E_2$,
respectively. This has a canonical map to $B$. Define $p:U \to \Hom_B(E_1,E_2)$ to be a plot if the composite
$U \to \Hom_B(E_1,E_2) \to B$ and the map \mbox{$U \times_B E_1 \to E_2$} defined by
$(u,x) \mapsto p(u)(x)$ are both smooth, where the pullback uses the first map.
It is straightforward to check that this makes $\Hom_B(\pi_1, \pi_2): \Hom_B(E_1,E_2) \to B$ into a diffeological vector pseudo-bundle.
Note that each fibre is exactly the set of all smooth linear maps $E_{1,b} \to E_{2,b}$ with the functional diffeology.
Given a third diffeological vector pseudo-bundle $\pi_3 : E_3 \to B$,
the composition map $\Hom_B(E_2,E_3) \times_B \Hom_B(E_1,E_2) \to \Hom_B(E_1,E_3)$ is
smooth, commutes with the maps to $B$, and is fibrewise bilinear. Hence it induces a bundle map
$\Hom_B(\pi_2,\pi_3) \otimes \Hom_B(\pi_1,\pi_2) \to \Hom_B(\pi_1,\pi_3)$.

We have that $\Hom_B(B \times V_1, B \times V_2) \cong B \times L^\infty(V_1,V_2)$.
Moreover, pullback along a smooth map $B' \to B$ commutes with forming hom bundles.
It follows that the hom bundle between two locally trivial bundles is again locally trivial.

Given a diffeological vector pseudo-bundle $\pi:E \to B$, we define its \textbf{dual}
$\pi^*:E^* \to B$ to be $\Hom(\pi, \Pr_1) : \Hom_B(E, B \times \R) \to B$,
where the codomain is the trivial bundle $\Pr_1:B \times \R \to B$.
Note that $\Hom_B(\Pr_1, \pi)$ is naturally isomorphic to $\pi$.

We have the following expected result:

\begin{prop}\label{prop:adjoint}
 Given a diffeological vector pseudo-bundle $\pi:E \to B$, we have an adjoint pair
 \[
  ? \otimes \pi : \DVPB_B \rightleftharpoons \DVPB_B: \Hom_B(\pi, ?).
 \]
 Moreover, given diffeological vector pseudo-bundles $\pi_1$, $\pi_2$ and $\pi_3$ over $B$,
 we have a natural isomorphism
 $\Hom_B(\pi_1 \otimes \pi_2, \pi_3) \cong \Hom_B(\pi_1, \Hom_B(\pi_2, \pi_3))$
 of diffeological vector pseudo-bundles over $B$.
\end{prop}
\begin{proof}
 For the first claim, we need to show that there is a natural bijection
 \[
  \DVPB_B(\pi_1 \otimes \pi, \pi_3) \cong \DVPB_B(\pi_1, \Hom_B(\pi,\pi_3)).
 \]
 Recall that there is a natural bijection from the left-hand side to the set of maps $\pi_1 \times_B \pi \to \pi_3$
 which are smooth and fibrewise bilinear. It is easy to check that the latter naturally bijects with the right-hand side.

 For the second claim, notice that for each $b \in B$, there is a bijection
 $L^\infty(E_{1b} \otimes E_{2b}, E_{3b}) \cong L^\infty(E_{1b}, L^\infty(E_{2b}, E_{3b}))$,
 which together give a bijection of the required sets. Using the definition of the plots of the $\Hom$ bundle,
 it is straightforward to check that this is indeed an isomorphism of diffeological vector pseudo-bundles over $B$.
\end{proof}

\begin{prop}\label{prop:hom-sum-prod}
Let $\pi_i : E_i \to B$ be a family of diffeological vector pseudo-bundles over $B$,
for $i$ in some indexing set, and
let $\pi : E \to B$ be another diffeological vector pseudo-bundle over $B$.
Then we have the following natural isomorphisms of diffeological vector pseudo-bundles over $B$:
\[
 \Hom_B(\bigoplus_i \pi_i, \pi) \cong \prod_{i,B} \Hom_B(\pi_i, \pi),
\]
\[
 (\bigoplus_i \pi_i) \otimes \pi \cong \bigoplus_i (\pi_i \otimes \pi),
\]
and
\[
 \Hom_B(\pi, \prod_{i,B} \pi_i) \cong \prod_{i,B} \Hom_B(\pi, \pi_i).
\]
\end{prop}

\begin{proof}
For the first claim, there is an evident natural smooth map from
left to right that one can check is an isomorphism.  %

For the second claim, the natural map goes from right to left,
and it follows from Proposition~\ref{prop:adjoint} that it is
an isomorphism, since $? \otimes \pi$ is a left adjoint.

For the third claim, the natural map goes from left to right,
and again it follows from Proposition~\ref{prop:adjoint} that it is
an isomorphism, since $\Hom_B(\pi, ?)$ is a right adjoint.
\end{proof}

Given diffeological vector pseudo-bundles $\pi_i:E_i \to B$ for $i = 1, 2$,
recall that $\DVPB_B(\pi_1,\pi_2)$ is a diffeological vector space under pointwise
operations, while $\Hom_B(\pi_1,\pi_2)$ has fibres which are diffeological vector spaces.
We can recover the former in terms of the latter.

\begin{lem}\label{lem:Hom-Gamma}
 There is an isomorphism $\DVPB_B(\pi_1,\pi_2) \cong \Gamma(\Hom_B(\pi_1, \pi_2))$
 as diffeological vector spaces.
\end{lem}

\begin{proof}
 A bundle map $f:E_1 \to E_2$ over $B$ induces a map $f_*:B \to \Hom_B(\pi_1,\pi_2)$ by sending $b$ to the restriction,
 $E_{1,b} \to E_{2,b}$ which is smooth and linear.
 It is straightforward to check that $f_*$ is a smooth section of $\Hom_B(E_1,E_2) \to B$.
 Therefore, we have a map $\DVPB_B(\pi_1,\pi_2) \to \Gamma(\Hom_B(\pi_1,\pi_2))$, and it is easy to see that this map is smooth
 and linear.

 On the other hand, given a section $s:B \to \Hom_B(E_1,E_2)$, we can define $s_*:E_1 \to E_2$ by sending
 $x$ to $s(\pi_1(x))(x)$. It is straightforward to see that $s_*$ is a bundle map over $B$. Hence, we have a map
 $\Gamma(\Hom_B(\pi_1,\pi_2)) \to \DVPB_B(\pi_1,\pi_2)$, which is easily seen to be smooth.

 Clearly these maps are inverses to each other, so that $\DVPB_B(\pi_1,\pi_2) \cong \Gamma(\Hom_B(\pi_1,\pi_2))$ as diffeological
 vector spaces.
\end{proof}

\subsubsection{The subbundle reachable by sections}\label{sss:E'}

Given a diffeological vector pseudo-bundle $\pi : E \to B$, it is not true in
general that every element of $E$ is the value of some section.
Write $\pi' : E' \to B$ for the subbundle whose total space consists of
$\{s(b) \mid s \in \Gamma(\pi),\, b \in B\}$.

\begin{lem}\label{lem:E'}
 Let $\pi:E \to B$ be a diffeological vector pseudo-bundle.
 Then the inclusion $E' \to E$ induces an isomorphism $\Gamma(\pi') \cong \Gamma(\pi)$
 of diffeological vector spaces.
\end{lem}

\begin{proof}
 This is straightforward.
\end{proof}

In Remark~\ref{rem:cross} and Example~\ref{ex:spaghetti-r} we give examples
of bundles where the inclusion $E' \to E$ is proper.

This phenomenon does not happen in classical differential geometry where
bundles are locally trivial and we have suitable bump functions on manifolds.
In diffeology, we can rule out this behaviour with similar assumptions.
Recall from~\cite[Definition~7.4]{CW5} that a \textbf{$D$-numerable diffeological
vector bundle} is a diffeological vector pseudo-bundle which is locally trivial
and which has a smooth partition of unity subordinate to a trivializing open cover of the base.

\begin{prop}\label{prop:trivial-cotangent}
 Let $\pi:E \to B$ be a $D$-numerable diffeological vector bundle with fibre type $V$.
 Then $E' = E$.
\end{prop}

\begin{proof}
 Let $\{B_i\}$ be a $D$-open cover of the base space $B$ so that $\pi|_{B_i}$ is trivial,
 and let $\{\rho_i:B \to \R\}$ be a smooth partition of unity
 subordinate to this cover. For any $b_0 \in B$ and any $x \in E_b$, we may choose $i$ so that $b_0 \in B_i$
 and $\rho_i(b_0) \neq 0$. Since $\pi|_{B_i}$ is trivial, there is a section $s$ of $\pi|_{B_i}$ with $s(b_0) = x$.
 Then $\frac{\rho_i}{\rho_i(b_0)} \cdot s:B \to E$ is a well-defined section of $\pi$ such that its value at $b_0$ is $x$.
 This implies that $E' = E$.
\end{proof}

\subsubsection{Module structures}

For any diffeological space $B$,
$C^\infty(B,\R)$ is a smooth commutative unital $\R$-algebra under pointwise operations.
Given a diffeological vector pseudo-bundle $\pi : E \to B$,
the diffeological vector space $\Gamma(\pi)$ of sections is naturally a
smooth $C^\infty(B,\R)$-module under pointwise multiplication and $\Gamma$
can be regarded as a functor from $\DVPB_B$ to the category of smooth $C^{\infty}(B,\R)$-modules.

More generally, for diffeological vector pseudo-bundles $\pi_1$ and $\pi_2$
over $B$, the diffeological vector space $\DVPB_B(\pi_1, \pi_2)$ is a
smooth $C^{\infty}(B,\R)$-module and the isomorphism
$\DVPB_B(\pi_1,\pi_2) \cong \Gamma(\Hom_B(\pi_1,\pi_2))$ of Lemma~\ref{lem:Hom-Gamma}
is an isomorphism of such modules.

For a smooth map $f : B' \to B$,
$f^*:C^\infty(B,\R) \to C^\infty(B',\R)$ is a smooth algebra homomorphism,
so that $\Gamma(\pi')$ can also be viewed as a smooth $C^\infty(B,\R)$-module,
for $\pi'$ a diffeological vector pseudo-bundle over $B'$.
With this structure, the map $\Gamma(\pi) \to \Gamma(f^*(\pi))$ described in
Section~\ref{ss:dvpb} is a smooth $C^\infty(B,\R)$-module homomorphism.

Similar remarks apply to other maps considered throughout the paper.

\section{Bundles of \texorpdfstring{$k$}{k}-forms}\label{se:duality}

In this section, we give a new definition of differential $k$-forms on a
diffeological space $X$ and compare it to the existing definition from~\cite{I2};
see Section~\ref{ss:cotangent} for a short summary.
We define the new differential forms to be
\[
  \widetilde{\Omega}^k(X) := \Gamma((\Exterior^k T^{\dvs} X)^* \to X),
\]
just as in ordinary differential geometry.
It is straightforward to check that $\widetilde{\Omega}^k$ is a contravariant functor
from $\Diff$ to the category of diffeological vector spaces.
We will see that there is a natural map $\phi : \widetilde{\Omega}^k(X) \to \Omega^k(X)$
which is an isomorphism when $k=0$ or $1$ and for all $k$ when $X$ is filtered.
It can fail to be injective and surjective in general.

There are simple situations in which the new forms contain more information.
For example, consider the inclusion $i$ of the axes $X_s$ into $\R^2$,
where the axes have the subset diffeology.
Let $\omega := i^*(dx \wedge dy)$ be the pullback of the volume form on $\R^2$.
On the one hand, when considered as a form in $\Omega^2(X_s)$, $\omega$ is zero.
To see this, it is enough to show that for any pointed plot $p:(U,0) \to (X_s,0)$,
we have $\omega_p(0) = 0$. This follows from the fact that for any $v_1,v_2 \in T_0(U)$,
$p_*(v_1)$ is parallel to $p_*(v_2)$ by definition of the diffeology on $X_s$.
On the other hand, when considered as a form in $\widetilde{\Omega}^2(X_s)$,
$\omega$ is non-zero at the origin, reflecting the fact that there
are tangent vectors $\ddx$ and $\ddy$ in $T_0^{\dvs}(X_s)$ such that
$\omega(\ddx, \ddy) = (dx \wedge dy)(i_* \ddx, i_* \ddy) = 1$.
The key difference is that the new definition considers tangent
vectors coming from \emph{different} plots, and thus sees $\omega$ to be non-zero.

In order to compare the two definitions, we first show in Section~\ref{ss:Omega}
that $\Omega^k(X)$ can be described as sections of the dual of a bundle we denote $T^k X$.
Then, in Section~\ref{ss:compare}, we compare $T^k X$ to $\Exterior^k T^{\dvs} X$,
showing that they agree when $X$ is filtered, or $k = 0$ or $1$.

\subsection{Differential forms as sections}\label{ss:Omega}

Let $X$ be a fixed diffeological space and let $k \geq 0$.
We define the diffeological vector pseudo-bundle $T^k X$ over $X$
to be the colimit in $\DVPB$ of the bundles $\Exterior^k TU \to U$
indexed by the plot category of $X$,
where $TU \to U$ is the usual tangent bundle and we identify
$\colim U$ with $X$.
When $k=1$, this gives the internal tangent bundle $T^{\dvs}X$,
by \cite[Theorem~4.17]{CW2}, and this special case can be seen
as a motivation for the above definition.

\begin{thm}\label{thm:Omega}
For any diffeological space $X$ and any $k \geq 0$, there is a natural isomorphism
\[
  \Gamma((T^kX)^* \to X) \lra \Omega^k(X)
\]
of diffeological vector spaces.
\end{thm}

\begin{proof}
 As usual, the diffeology on the space of sections is the subset diffeology from the
 space of all maps from $X$ to the total space,
 and $\Omega^k(X)$ is defined to be $\lim \Omega^k(U)$ as
 diffeological vector spaces; see the beginning of Section~\ref{ss:cotangent}.

 The isomorphism is given by the following composite:
 \begin{align*}
  &\Gamma((T^kX)^* \to X) \\
  \cong\; &\DVPB_X(T^kX \to X,\, X \times \R \to X) \\
  =\; &\DVPB_X(\colim (\Exterior^k TU \to U),\, X \times \R \to X) \\
  \cong\; &\{f \in \lim \DVPB(\Exterior^k TU \to U,\, X \times \R \to X) \mid
  f_p \text{ lies over the plot } p : U \to X\} \\
  \cong\; &\lim \DVPB_U(\Exterior^k TU \to U,\, U \times \R \to U) \\
  \cong\; &\lim \Gamma((\Exterior^k TU)^* \to U)
  = \lim \Omega^k(U)
  = \Omega^k(X).
 \end{align*}
 The first isomorphism is Lemma~\ref{lem:Hom-Gamma}.
 The isomorphism in the fourth line is Proposition~\ref{prop:colimit-limit}.
 For each plot $p : U \to X$, giving the component $f_p$ of $f$
 is equivalent to giving a map from $\Exterior^k TU$ to $U \times \R$
 over the identity map, which gives the isomorphism in the fifth line.
 This isomorphism is easily checked to be a diffeomorphism.
 The first isomorphism in the last line is again Lemma~\ref{lem:Hom-Gamma}.
 The last two equalities are true by definition.
\end{proof}

\begin{rem}
 We now describe the fibres of the bundle $(T^k X)^* \to X$.
 By definition, the fibre over $x \in X$ is the smooth dual of $T^k_x (X)$.
 The construction of colimits in Theorem~\ref{thm:DVPB-bicomplete} shows
 that the latter is isomorphic to $\colim \Exterior^k T_0(U)$,
 where the colimit is taken over the category of pointed plots $(U,0) \to (X,x)$.
 The diffeology on the colimit is complicated, but we can say that the
 smooth dual of $T^k_x(X)$, as a vector space, is a subspace of
 $\Vect(\colim \Exterior^k T_0(U),\, \R) \cong \lim \Vect(\Exterior^k T_0(U),\, \R)$.
 That is, an element of $(T^k_x(X))^*$ is a compatible family of
 $k$-forms at $0 \in U$, indexed by the pointed plots $(U,0) \to (X,x)$,
 subject to a smoothness condition.
\end{rem}

\begin{rem}\label{rem:Uplambda-Exterior}
In~\cite[Chapter 6]{I2}, it is shown that $\Omega^k(X) \cong \Gamma(\Uplambda^k X)$,
where $\Uplambda^k X$ is the bundle described in Section~\ref{ss:cotangent}.
We now show that there is an injective bundle map $\Uplambda^k X \to (T^k X)^*$.
As a vector space,
the fibre $\Uplambda^k_x (X)$ of $\Uplambda^k X$ over a point $x \in X$
is defined to be the quotient of $\Omega^k(X)$ where two $k$-forms are identified
if they agree at the origin when pulled back along any pointed plot $p:(U,0) \to (X,x)$.
This is the same as saying that they agree when viewed as linear maps
$\Exterior^k T_0(U) \to \R$ for each such $p$.
By the description of the fibres of $T^k X$ in the previous remark,
it follows that $\Uplambda^k_x (X)$ is isomorphic as a vector space to the image of the restriction map
$\Omega^k(X) \to (T^k_x (X))^*$.
In fact, we have an injective bundle map $\Uplambda^k X \to (T^k X)^*$ whose
image is precisely the subbundle reachable by sections (see Section~\ref{sss:E'}).
In particular, this map induces a bijection on sections.
The difference between the two bundles is that in $(T^k X)^*$ one considers
$k$-forms defined at just one point, while in $\Uplambda^k X$ one
only considers forms that extend globally to $X$.
In Example~\ref{ex:spaghetti-r}, we will see that these differ in general.
Allowing pointwise forms gives additional flexibility.

We don't know whether $\Uplambda^k X$ has the subset diffeology from $(T^k X)^*$,
i.e., whether $\Uplambda^k X$ is diffeomorphic to the subbundle reachable by
sections.
\end{rem}

\subsection{Comparing the two definitions of forms}\label{ss:compare}

We now compare the two definitions of differential forms.
To do so, we define a comparison map of the underlying bundles.

Let $X$ be a fixed diffeological space.
Given a plot $p : U \to X$, there is a natural map
\[
\xymatrix{TU \ar[r]^-{Tp} \ar[d] & T^{\dvs} X \ar[d] \\ U \ar[r]_p & X}
\]
of diffeological vector pseudo-bundles.
This induces a natural bundle map
\[
\xymatrix @C+0.7pc {\Exterior^k TU \ar[r]^-{\Exterior^k Tp} \ar[d] & \Exterior^k T^{\dvs} X \ar[d] \\ U \ar[r]_p & X}
\]
and therefore a canonical bundle map $\rho : T^k X \to \Exterior^k T^{\dvs} X$ over $X$.
The map $\rho$ induces a map
\[
  \rho^* : (\Exterior^k T^{\dvs} X)^* \lra (T^k X)^*
\]
of dual bundles over $X$, which induces a map
\[
  \phi : \widetilde{\Omega}^k(X) \lra \Omega^k(X)
\]
of diffeological vector spaces,
using the definition of $\widetilde{\Omega}^k(X)$ and Theorem~\ref{thm:Omega}.

We can now prove the main result of this section.

\begin{thm}\label{thm:recover-forms}
 Let $X$ be a diffeological space and let $k$ be a natural number.
 If $k = 0$ or $1$, or $X$ is filtered, then the map
 \[
  \rho : T^k X \to \Exterior^k T^{\dvs} X
 \]
 is an isomorphism in $\DVPB_X$ and the map
 \[
  \phi : \widetilde{\Omega}^k(X) \lra \Omega^k(X)
 \]
 is an isomorphism of diffeological vector spaces.
\end{thm}

The case where $k = 1$ is also proved in~\cite[Proposition~4.6]{MW}.

\begin{proof}
 It is clear that when $\rho$ is an isomorphism, so is $\phi$,
 using Theorem~\ref{thm:Omega}.
 So we just need to prove the first claim.

 Let $k = 0$.  Then $\Exterior^0 TU = U \times \R$, and from the
 description of colimits in $\DVPB$ given in Theorem~\ref{thm:DVPB-bicomplete},
 one sees that $T^0 X = X \times \R$.
 By definition, $\Exterior^0 T^{\dvs} X = X \times \R$.
 It is not hard to see in this case that $\rho : T^0 X \to \Exterior^0 T^{\dvs} X$
 is an isomorphism.

 Let $k = 1$.  Then the claim that $\rho$ is an isomorphism is saying
 that $T^{\dvs}X$ is the colimit in $\DVPB$ of the bundles $TU \to U$,
 which is \cite[Theorem~4.17]{CW2}.

 Finally, suppose that $X$ is filtered.
 We will show that $\rho$ is a bijection, and that its inverse is also smooth.
 For the first step, we will show that $\rho$ is a bijection on each fibre.
 Let $x \in X$.
 Recall that $T^k_x(X)$ is the colimit of $\Exterior^k T_0(U)$ over all pointed plots $(U,0) \to (X,x)$
 in the germ category $\mathcal{G}(X,x)$; see Theorem~\ref{thm:DVPB-bicomplete}.
 Since this indexing category is filtered, it suffices to show that the exterior
 power functor $\Exterior^k : \Vect \to \Vect$ preserves filtered colimits.
 By~\cite[Corollary~3.13 and Examples~3.19(2) and~3.23(4)]{AMSW}, it suffices to show
 that for every $V \in \Vect$ and every finite-dimensional subspace $W$ of $\Exterior^k V$,
 there exists a finite-dimensional subspace $V'$ of $V$ such that $W$ is
 contained in the subspace $\Exterior^k V'$ of $\Exterior^k V$.
 Given such a $W$, it is the span of a finite list of vectors in $\Exterior^k V$,
 each of which involves only finitely many vectors in $V$.
 Taking $V'$ to be the span of all of the vectors involved ensures
 that $\Exterior^k V'$ contains $W$.

 We now show that $\rho^{-1}$ is smooth, or equivalently, that $\rho$ is a subduction.
 Suppose given a plot $q : V \to \Exterior^k TX$.
 Without loss of generality, we'll assume that $0 \in V$, and we'll show that
 $\rho^{-1} \circ q$ is smooth in a neighbourhood of $0$.
 That is, we'll show that $q$ is locally of the form $\rho \circ r$ for some plot $r$
 in $T^kX$.
 Let $q' = \pi^k_X \circ q : V \to X$ and let $x_0 = q'(0)$.
 Each time we work locally, we'll implicitly reduce the size of $V$.
 Since $\Exterior^k TX$ is a quotient of $F_X(\prod_{k,X} TX)$, $q$ can be expressed
 locally as a finite sum $\sum_i r_i(v) \, q_{i,1}(v) \wedge \cdots \wedge q_{i,k}(v)$,
 where $r_i : V \to \R$ and $q_{i,j} : V \to TX$ are smooth, and $\pi_X \circ q_{i,j} = q'$,
 using Proposition~\ref{pr:free}.
 Since we are forming a sum, we can assume that each $q_{i,j}$ is in Hector's
 diffeology, so locally each $q_{i,j}$ factors as $Tp_{i,j} \circ f_{i,j}$ for some plots
 $p_{i,j} : U_{i,j} \to X$ and smooth maps $f_{i,j} : V \to TU_{i,j}$.
 Without loss of generality, we assume that $\pi_{U_{i,j}}(f_{i,j}(0)) = 0 \in U_{i,j}$.
 It follows that the plots $p_{i,j} : (U_{i,j},0) \to (X,x_0)$ are pointed.
 Since $X$ is filtered, these plots all locally factor through a common plot
 $p : (U, 0) \to (X, x_0)$ in a way that coequalizes the germs $\pi_{U_{i,j}} \circ f_{i,j}$.
 That is, there exist smooth maps $g_{i,j}:U_{i,j} \to U$ such that $p_{i,j} = p \circ g_{i,j}$
 (and hence $Tp_{i,j} = Tp \circ Tg_{i,j}$) and the composite $g_{i,j} \circ \pi_{U_{i,j}} \circ f_{i,j}$
 is independent of the subscripts $i,j$, all as germs.
 Therefore, $q = \rho \circ r$, where the germ $r:V \to T^k X$ around $0$ is defined by 
 \[
  r(v) = \iota_U\big( \sum_i r_i(v) (Tg_{i,1}(f_{i,1}(v))) \wedge \cdots \wedge (Tg_{i,k}(f_{i,k}(v))) \big),
 \]
 which is clearly a plot of $T^k X$. Here $\iota_U:\Exterior^k TU \to T^k X$ is the map expressing $T^k X$ as a colimit.
\end{proof}

A similar result holds for tensor powers, using a similar proof.

\begin{rem}\label{rem:weak}
 One can show that when $X$ is only assumed to be weakly filtered,
 it still follows that the bundle map
 $\rho:T^k X \to \Exterior^k T^{\dvs} X$ is surjective.
 Hence, in this case, $\rho^*:(\Exterior^k T^{\dvs}X)^* \to (T^k X)^*$ is injective, and therefore,
 the canonical map $\phi : \widetilde{\Omega}^k(X) \to  \Omega^k(X)$ is injective.
 We show in Example~\ref{ex:R2Z2} that when $X$ is
 weakly filtered, $\rho$ is not necessarily injective
 and $\phi$ is not necessarily surjective.

 Moreover, we show in Example~\ref{ex:Xs} that when $X$ is not weakly filtered,
 $\rho$ is not necessarily surjective and $\phi$ is not necessarily injective.
 We give another such example in Propositions~\ref{prop:not-surjective} and~\ref{prop:Phi-not-isom}.
\end{rem}

\begin{ex}\label{ex:R2Z2}
 Let the cyclic group $\Z_2$ act on $\R^2$ by negation, and write $X$ for the orbit space with the quotient diffeology.
 Then, by an argument similar to~\cite[Example~4.7]{CW3},
 $X$ is weakly filtered at $[0]$ but not filtered at $[0]$,
 where $[0]$ the image of the origin of $\R^2$.

 To show that $\rho : T^2 X \to \Exterior^2 T^{\dvs} X$ is not injective,
 we'll show that the vector space $\Exterior^2 T_{[0]}(X) = 0$ while $T^2_{[0]} (X) \neq 0$.
 First note that the quotient map $q : \R^2 \to X$ induces a surjection
 $T_0(\R^2) \to T_{[0]}(X)$, since pointed plots to $(X, 0)$ locally lift to $(\R^2, 0)$.
 Thus any element $v \in T_{[0]}(X)$ is of the form $(q \circ p)_*(\ddt)$
 where $p : \R \to \R^2$ is a linear plot.
 Since $q \circ p = q \circ (-p)$, we deduce that $v = -v$, so $v = 0$.
 It follows that $\Exterior^2 T_{[0]}(X) = 0$.

 Next we show that $T^2_{[0]} (X) \neq 0$.
 We know from~\cite[9.32]{I2} (see also~\cite{KW}) that $q$ induces an injection
 $q^* : \Omega^2(X) \to \Omega^2(\R^2)$ with image
 $\{f(x,y) \, dx \wedge dy \mid f \text{ is smooth and $\Z_2$-invariant} \}$.
 Let $\omega \in \Omega^2(X)$ be such that $q^*(\omega) = dx \wedge dy$.
 Regard $\omega$ as a bundle map $T^2 X \to X \times \R$ over $X$, using
 that $\Omega^2(X) \cong \DVPB_X(T^2 X, X \times \R)$,
 as in the proof of Theorem~\ref{thm:Omega}.
 Then $q^*(\omega)$ is represented by the bundle map $T^2 \R^2 \to \R^2 \times \R$ over $\R^2$
 induced by the composite $T^2\R^2 \to T^2X \to X \times \R$.
 Now $T^2_0(\R^2)$ is isomorphic to $\Exterior^2 T_0(\R^2)$
 and $dx \wedge dy$ is non-zero on this fibre,
 so it follows that $T^2_{[0]}(X)$ must be non-zero.
 Therefore, $\rho$ is not injective.

 We now show that the map $\phi : \widetilde{\Omega}^2(X) \to  \Omega^2(X)$
 is not surjective by showing that $\omega$ defined above is not in the image.
 If it were, then there would be a bundle map $\omega'$ over $X$ making the following
 diagram commute:
 \[
  \xymatrix{
    T^2 X \ar[r]^{\omega} \ar[d]_{\rho} & X \times \R \\
    \Exterior^2 T^{\dvs} X \ar[ur]_{\omega'} .
  }
 \]
 But $\omega$ is non-zero on the fibre $T^2_{[0]}(X)$ and $\Exterior^2 T_{[0]}(X) = 0$,
 so this is not possible.
\end{ex}

\begin{ex}\label{ex:Xs}
As in the introduction to Section~\ref{se:duality},
write $X_s$ for the axes in $\R^2$ with the subset diffeology.
Note that $X_s$ is not weakly filtered.
Let $\omega := i^*(dx \wedge dy) \in \widetilde{\Omega}(X_s)$,
where $i : X_s \to \R^2$ denotes the inclusion.
The earlier calculation showed that $\omega$ is non-zero but $\phi(\omega)$ is zero.
Therefore, $\phi$ is not injective.
It then follows immediately that $\rho^*$ is not injective
and that $\rho$ is not surjective.
\end{ex}

\section{The axes in \texorpdfstring{$\R^2$}{R2}}\label{se:axes}

In this section, we focus on the diffeological space $X_g$, the axes in
$\R^2$ with the gluing diffeology.  In other words, $X_g$ is
the quotient space of $(\R \times \{0\}) \coprod (\{0\} \times \R)$
with the two copies of the origin identified.
We consider it as a subset of $\R^2$, but it has fewer plots than $X_s$,
the axes with the subset diffeology (Example~\ref{ex:Xs}).
We give complete descriptions of $T^{\dvs} X_g$, vector fields on $X_g$,
and $\Omega^1(X_g)$, and prove the results mentioned in Remark~\ref{rem:weak}.
The methods illustrate the usefulness of the material in Section~\ref{se:operations}.

We know from~\cite[Example~3.17]{CW2}
that the tangent space of $X_g$ at the origin is $2$-dimensional, while the other tangent spaces are $1$-dimensional.
Moreover, not every tangent vector at the origin is $1$-representable, so $X_g$ is not weakly filtered.
(Recall that a tangent vector $v$ in the tangent space $T_x(X)$ is $1$-representable if there exists a pointed plot 
$p:(\R,0) \to (X,x)$ such that $v = p_*(\ddt)$. This is not always the case for a 
general diffeological space~\cite[Remark~3.8]{CW2}, but is true for every weakly filtered space~\cite[Proposition~4.11]{CW3}.)
By~\cite[Example~4.19(2)]{CW2}, each tangent space has the fine diffeology.

We begin by giving an explicit description of the tangent bundle $\pi : T^{\dvs}X_g \to X_g$.
We make use of the following diffeological vector pseudo-bundles.
Let $E_1$ be the gluing of $(\R \times \{0\}) \times \R$ and
$(\{0\} \times \R) \times \{0\}$ at $(0,0,0)$, regarded as a subset of $\R^3$,
and define $\pi_1 : E_1 \to X_g$ to be the map sending $(x,y,v)$ to $(x,y)$.
This is a diffeological vector pseudo-bundle which has $1$-dimensional fibres
over the $x$-axis in $X_g$, and $0$-dimensional fibres elsewhere.
Similarly, let $E_2$ be the gluing of
$(\R \times \{0\}) \times \{0\}$ and $(\{0\} \times \R) \times \R$ at $(0,0,0)$,
with $\pi_2 : E_2 \to X_g$ sending $(x,y,v)$ to $(x,y)$.

\begin{prop}\label{prop:tangent-as-sum}
 $T^{\dvs} X_g \cong E_1 \oplus E_2$ as diffeological vector pseudo-bundles over $X_g$.
\end{prop}
\begin{proof}
 Since the inclusion $i:\R \to X_g$ of the $x$-axis is smooth, we
 have a smooth map $T\R \to T^{\dvs} X_g$.
 Since the inclusion $j:\R \to X_g$ of the $y$-axis is smooth, the composite
 \[
  \xymatrix{\R \ar[r]^-j & X_g \ar[r]^-\sigma & T^{\dvs} X_g}
 \]
 is smooth, where $\sigma$ is the zero section. These together induce a bundle map $\pi_1:E_1 \to T^{\dvs} X_g$ over $X_g$.
 Similarly we have another bundle map $\pi_2:E_2 \to T^{\dvs} X_g$ over $X_g$. Therefore, we have a bundle map
 $\psi : E_1 \oplus E_2 \to T^{\dvs} X_g$ over $X_g$.
 It is easy to see that $\psi$ is a smooth bijection.
 To see that $\psi$ is an isomorphism of diffeological vector pseudo-bundles over $X_g$,
 we must show that it is a subduction.
 Since $E_1 \oplus E_2$ is a diffeological vector pseudo-bundle, it is enough
 to check that every plot in Hector's diffeology lifts through $\psi$.
 Every such plot $q : V \to T^{\dvs}X_g$ locally factors through a plot of the
 form $Tp : T\R \to T^{\dvs}X_g$, where $p$ is the inclusion of one of the axes in $X_g$,
 since these inclusions generate the diffeology on $X_g$.
 But this is the same as saying that $q$ locally factors through one of the
 maps $E_i \to T^{\dvs}X_g$, which implies that $\psi$ is a subduction.
\end{proof}

\begin{rem}
 More generally, let $\{(X_i,x_i)\}_{i \in I}$ be a family of pointed diffeological spaces.
 Write $X$ for the wedge of the $X_i$'s, that is, the gluing of the $X_i$'s
 along their distinguished points.
 Consider $T^{\dvs} X_i$ to be pointed by the zero vector in $T_{x_i}X_i$,
 and write $T_i$ for the wedge of $T^{\dvs} X_i$ with all $X_j$'s for $j \neq i$.
 Then $T_i$ is a diffeological vector pseudo-bundle over $X$, with $0$-dimensional
 fibres away from $X_i$.
 The natural map $\bigoplus_{i \in I} T_i \to T^{\dvs} X$ is an
 isomorphism of diffeological vector pseudo-bundles over $X$.
 The proof that this map is a subduction is similar to the proof of
 Proposition~\ref{prop:tangent-as-sum}.
 The fact that the map is injective follows from the fact that the
 inclusions $X_i \to X$ have retractions.
\end{rem}

\begin{prop}\label{section-t}
 For the tangent bundle $\pi : T^{\dvs}X_g \to X_g$, we have
 \begin{align*}
   \Gamma(\pi)
    &= \{ f(x) \ddx + g(y) \ddy \mid f, g \in C^{\infty}(\R, \R) \text{ with } f(0) = 0 = g(0) \}\\
    &\cong V \oplus V ,
 \end{align*}
 where $f(x) \ddx + g(y) \ddy$ denotes the section sending $(x, 0)$ to $f(x) \ddx$
 and $(0,y)$ to $g(y) \ddy$,
 $V := \mbox{$\{ f \in C^{\infty}(\R, \R) \mid f(0) = 0 \}$}$ is equipped with the subset diffeology,
 and the isomorphism is as diffeological vector spaces.
\end{prop}

Note that while $\ddx$ denotes a well-defined tangent vector in $T_{(x,0)}X_g$
for each $x \in \R$, $\ddx$ does not extend to a smooth section of $\pi$.

\begin{proof}
 Since $\pi \cong \pi_1 \oplus \pi_2$, we have an isomorphism
 $\Gamma(\pi) \cong \Gamma(\pi_1) \oplus \Gamma(\pi_2)$ as diffeological vector spaces.
 Due to the nature of the gluing diffeology, all sections of $\pi_1$ and $\pi_2$
 must vanish at the origin and we can directly compute that
 $\Gamma(\pi_1) \cong V \cong \Gamma(\pi_2)$, giving the claimed isomorphism.
 Using the isomorphism, we then see that
 \[
   \Gamma(\pi)
    = \{ f(x) \ddx + g(y) \ddy \mid f, g \in C^{\infty}(\R, \R) \text{ with } f(0) = 0 = g(0) \} ,
 \]
 as required.
\end{proof}

\begin{rem}\label{rem:cross}
 In particular, while the vector space $T_0(X_g)$ is 2-dimensional, all sections vanish there.
 In the notation of Section~\ref{sss:E'}, $(T^{\dvs} X_g)'$ is the bundle which
 is $0$-dimensional over the origin and the same as $T^{\dvs} X_g$ elsewhere.
 As stated in Lemma~\ref{lem:E'}, this inclusion induces an isomorphism on sections.
 A similar phenomenon happens with the bundles $\pi_1$ and $\pi_2$ defined above.
\end{rem}

We next use this calculation to compute the space of differential forms on $X_g$.

\begin{prop}\label{prop:form}
 We have
 \begin{align*}
   \Omega^1(X_g)
    &= \{ f(x) dx + g(y) dy \mid f, g \in C^{\infty}(\R, \R) \} \\
    &\cong C^{\infty}(\R, \R) \times C^{\infty}(\R, \R) ,
 \end{align*}
 where $dx$ is trivial away from the $x$-axis and $dy$ is trivial away from the $y$-axis,
 and the isomorphism is as diffeological vector spaces.
\end{prop}

More formally, $dx$ and $dy$ are defined to be the pullbacks of
$dx$ and $dy$ in $\Omega^1(\R^2)$ along the inclusion $X_g \to \R^2$.

\begin{proof}
 By Theorem~\ref{thm:Omega}, $\Omega^1(X_g) \cong \Gamma((T^{\dvs} X_g)^* \to X_g)$, so we compute the latter.
 By Propositions~\ref{prop:hom-sum-prod} and~\ref{prop:tangent-as-sum}, we have
 $(T^{\dvs} X_g)^* \cong E_1^* \times_{X_g} E_2^*$ as diffeological vector pseudo-bundles over $X_g$.
 Therefore,
 \[
    \Gamma((T^{\dvs} X_g)^* \to X_g) \cong \Gamma(E_1^* \to X_g) \times \Gamma(E_2^* \to X_g) .
 \]
 To calculate $\Gamma(E_1^* \to X_g)$, fix $\ell \in L^\infty(\R,\R)$ and consider the map $p_{\ell}:\R \to E_1^*$ defined by
 \begin{equation*}%
  t \mapsto \begin{cases} (0,t,0^*), & \textrm{if $t \neq 0$} \\ (0,0,\ell), & \textrm{if $t=0$,} \end{cases}
 \end{equation*}
 where $0^* \in L^\infty(\R,\R)$ denotes the zero map, and we express elements in $E_1^*$ as elements of the fibrewise dual.
 We claim that $p_{\ell}$ is smooth.
 By the definition of the diffeology on $E_1^* = \Hom_{X_g}(E_1, X_g \times \R)$
 (see Section~\ref{sss:Hom}),
 we must check two things.
 First observe that the composite $\R \to E_1^* \to X_g$ sends $t$ to $(0,t)$,
 so it is smooth.
 Next we must show that a certain map $f : \R \times_{X_g} E_1 \to X_g \times \R$ is smooth.
 Here, the pullback is taken along the map $\R \to X_g$ sending $t$ to $(0,t)$,
 and $f$ is defined by
 \[
  (t,(0,t,v)) \mapsto ((0,t),p_{\ell}(t)(0,t,v)) = \begin{cases} ((0,t),0), & \textrm{if $t \neq 0$} \\ ((0,0),\ell(v)), & \textrm{if $t = 0$} . \end{cases}
 \]
 To see that $f$ is smooth, let $(\alpha,\beta): U \to \R \times_{X_g} E_1$ be a plot
 and consider the composite $U \to \R \times_{X_g} E_1 \to X_g \times \R$.
 The component of the composite landing in $X_g$ sends $u$ to $(0, \alpha(u))$,
 and so is smooth.
 We are left to show that the other component is smooth.
 Since $E_1$ has the gluing diffeology,
 $\beta$ locally factors through either $\R \times \{0\} \times \R$ or $\{0\} \times \R \times \{0\}$. In the first case,
 $\alpha$ is constant at zero, and hence the composite is given by
 $u \mapsto \ell(\Pr_3(\beta(u)))$, which is smooth.
 In the second case, the third component of $\beta$ is constant at zero, and hence the composite is given by
 the constant map $0$, which is again smooth.

 It follows from the above that any smooth section of $E_1^*$ over the $x$-axis
 can be extended by zero to give a smooth global section of $E_1^*$.
 Therefore, $\Gamma(E_1^* \to X_g) \cong C^{\infty}(\R, \R)$ as diffeological vector
 spaces.
 All of the above goes through for $E_2$ as well, so we conclude that
 \[
 \begin{aligned}
   \Omega^1(X_g)
   &\cong \Gamma((T^{\dvs} X_g)^* \to X_g) \\
   &\cong \Gamma(E_1^* \to X_g) \times \Gamma(E_2^* \to X_g) \\
   &\cong C^{\infty}(\R, \R) \times C^{\infty}(\R, \R) .
 \end{aligned}
 \]
 This implies that
 \[ \Omega^1(X_g) = \{ f(x) dx + g(y) dy \mid f, g \in C^{\infty}(\R, \R) \}, \]
 as required.
\end{proof}

Note that the proof shows that $(T^{\dvs}X_g)^* \to X_g$ has sections which are non-zero
at the origin, in contrast to how sections of $T^{\dvs}X_g \to X_g$ behave,
as explained in Proposition~\ref{section-t}.
Note that for each $x \in X_g$, the dimensions of the fibres of these two bundles at $x$
agree, showing that the difference arises from the diffeologies.

\begin{rem}
 We give another way to calculate $\Omega^1(X_g)$. 
 By Proposition~\ref{prop:top-forms}, a $1$-form on $X_g$ is equivalent to a compatible 
 family of $1$-forms on the domains of all of the $1$-plots.
 Every plot locally factors through either the inclusion of the $x$-axis or the
 inclusion of the $y$-axis, and this factorization is unique, except where the
 plot is locally constant at the origin.
 It follows that in order to assign a compatible family of $1$-forms to the
 domains of the $1$-plots, it is equivalent to choose $1$-forms on the $x$-axis
 and the $y$-axis.
 (At a place where a plot is locally constant, either choice of factorization will
 assign the zero $1$-form.)
 Therefore, we have $\Omega^1(X_g) \cong \Omega^1(\R) \times \Omega^1(\R)$, 
 which matches the calculation in Proposition~\ref{prop:form}.
\end{rem}

We next give another example showing that $\rho$ can fail to be surjective
when $X$ is not weakly filtered.
(Cf.\ Theorem~\ref{thm:recover-forms} and Example~\ref{ex:Xs}.)

\begin{prop}\label{prop:not-surjective}
 The canonical map $\rho:T^2 X_g \to \Exterior^2 \, T^{\dvs} X_g$,
 introduced in Section~\ref{se:duality}, is not surjective.
\end{prop}
\begin{proof}
 The fibre of $\Exterior^2 \, T^{\dvs} X_g \to X_g$ at the origin is
 $\Exterior^2 \, T_0^{\dvs}(X_g)$, which is $1$-dimensional since $T_0^{\dvs}(X_g) = \R^2$.
 On the other hand, all of the fibres of $T^2 X_g \to X_g$ are $0$-dimensional,
 by the following result.
 Therefore, the canonical map $\rho:T^2 X_g \to \Exterior^2 \, T^{\dvs} X_g$ cannot be surjective.
\end{proof}

 Note that $X_g$ has $D$-dimension $1$.

 \begin{prop}\label{prop:dim}
  Let $X$ be a diffeological space of $D$-dimension $n$. Then for $k > n$, $T^k X \to X$ is a diffeomorphism
  and $\Omega^k(X) = 0$.
 \end{prop}

 \begin{proof}
  By the description of colimits in $\DVPB$ in the proof of Theorem~\ref{thm:DVPB-bicomplete}, $T^k_x (X) \cong \colim \Exterior^k  T_0(U)$,
  where the colimit is over pointed plots $(U, 0) \to (X, x)$.
  By assumption, each such plot factors in the germ category through
  a plot $(V, 0) \to (X, x)$ with $\dim(V) \leq n < k$.
  Since $\Exterior^k T_0(V) = 0$, the colimit vanishes.
  Therefore, $T^k X \to X$ is a smooth bijection.  The zero section provides
  a smooth inverse, showing that it is a diffeomorphism.

  The claim about $\Omega^k(X)$ follows from this and~Theorem~\ref{thm:Omega},
  and is also proved in~\cite[6.39]{I2}.
 \end{proof}

We end this section with another example showing that $\phi$ can fail to be injective
when $X$ is not weakly filtered.
(Cf.\ Theorem~\ref{thm:recover-forms} and Example~\ref{ex:Xs}.)

\begin{prop}\label{prop:Phi-not-isom}
 The diffeological vector space $\widetilde{\Omega}^2(X_g)$ is $1$-dimensional
 while $\Omega^2(X_g) = 0$.
 In particular, there is no injective linear map $\widetilde{\Omega}^2(X_g) \to \Omega^2(X_g)$.
\end{prop}

\begin{proof}
 We first calculate $\widetilde{\Omega}^2(X_g)$, which by definition is
 $\Gamma((\Exterior^2 T^{\dvs} X_g)^* \to X_g)$. Since $T^{\dvs} X_g \cong E_1 \oplus E_2$
 by Proposition~\ref{prop:tangent-as-sum},  we have
 \[
  T^{\dvs} X_g \otimes T^{\dvs} X_g \cong (E_1 \otimes E_1) \oplus (E_1 \otimes E_2)
  \oplus (E_2 \otimes E_1) \oplus (E_2 \otimes E_2),
 \]
 and hence $\Exterior^2 T^{\dvs} X_g \cong E_1 \otimes E_2$, all of these isomorphisms as diffeological vector pseudo-bundles over $X_g$.
 Note that for any $\ell \in (\R \otimes \R)^*$, the map $X_g \to (E_1 \otimes E_2)^*$ defined by
 \[
   (x,y) \mapsto \begin{cases}
                (x,y,0^*), & \textrm{if $(x,y) \neq (0,0)$} \\
                (0,0,\ell), & \textrm{else}
              \end{cases}
 \]
 is smooth, which can be proved by an argument similar to that in the proof of Proposition~\ref{prop:form}. This then implies that
 $\Gamma((\Exterior^2 T^{\dvs} X_g)^* \to X_g)$ is $1$-dimensional.
 On the other hand, by Proposition~\ref{prop:dim}, we know that $\Omega^2(X_g) = 0$.
 The claim follows.
\end{proof}

\section{\texorpdfstring{$\R^2$}{R2} with the spaghetti diffeology}
\label{se:spaghetti}

In this section, we study the space $\R^2_S$, which is the set $\R^2$ with the spaghetti diffeology
(also known as the wire diffeology~\cite[1.10]{I2}),
i.e., the diffeology generated by all smooth maps $\R \to \R^2$.
We begin with a result about the tangent spaces, filling in some details
from the claim made in~\cite[Example~3.22(1)]{CW2}.

\begin{prop}
For each $x \in \R^2_S$, the tangent space $T^{\dvs}_x(\R^2_S)$ has uncountable
dimension.
\end{prop}

\begin{proof}
Since the translation map $\R^2_S \to \R^2_S$ sending $0$ to $x$ is a diffeomorphism,
it is enough to consider the case where $x = 0$.
For each $m \in \R$, consider the pointed curve 
$p_m:(\R,0) \to (\R^2_S,0)$ sending $t$ to $(t,mt)$.
This represents a nonzero element $(p_m)_*(\ddt)$ in $T^{\dvs}_0 (\R^2_S)$, since its
image under the map $T^{\dvs}_0 (\R^2_S) \to T_0(\R^2)$ induced by the
identity map $\R^2_S \to \R^2$ is nonzero.
We claim that these tangent vectors are linearly independent.
Indeed, the linear relations between tangent vectors are generated by
factorizations of one plot through another,
as described in the proof of~\cite[Proposition~3.4]{CW2}.
Since the plots of $\R^2_S$ all factor through curves
and the plots $p_m$ all have distinct images, even locally,
we deduce that the set $\{ (p_m)_*(\ddt) \}_{m \in \R}$ is linearly independent.
\end{proof}

Now we characterize the diffeology on each tangent space:

\begin{prop}\label{prop:fine}
For $x \in \R^2_S$, $T_x^{\dvs}(\R^2_S)$ is a fine diffeological vector space.
In other words, every linear map $T_x^{\dvs}(\R^2_S) \to \R$ is smooth.
\end{prop}

\begin{proof}
The equivalence of the two claims is~\cite[Proposition~3.4]{CW4}.
Without loss of generality, we take $x = 0$.
Let $\alpha : T_0^{\dvs}(\R^2_S) \to \R$ be a linear map.
Since dvsification is left adjoint to the forgetful functor (see~\cite[Section~4.2]{CW2}),
it suffices to show that $\alpha$ is smooth on the tangent space
$T^H_0(\R^2_S)$ with Hector's diffeology.
By Boman's theorem (\cite[Corollary~3.14]{KM}), it is enough to test this on plots $\R \to T_0^H (\R^2_S)$.
Every such plot locally factors as
\[
  \R \llra{q} T\R \llra{Tp} T^H \R^2_S ,
\]
where $p : \R \to \R^2_S$ is a plot and the composite lands in $T_0 ^H(\R^2_S)$.
Let $f = \pi \circ q$, where $\pi : T\R \to \R$ is the projection.
We must have that $p \circ f = 0 \in \R^2_S$.
Let $Z$ be the path component of $p^{-1}(0)$ containing the image of $f$.
\smallskip

\noindent
\textbf{Case 1:}  $Z = \{ x \}$, a single point.
Then $f$ is constant.  It follows that $q$ lands in $T_x (\R)$, and
so the composite map factors as
\[
  \R \llra{q} T_x(\R) \llra{p_*} T_0 ^H(\R^2_S) \llra{\alpha} \R .
\]
The second and third maps are linear, so their composite is smooth.
So the whole composite is smooth.
\smallskip

\noindent
\textbf{Case 2:}  $Z$ is a non-trivial interval.
Since $p$ is constant on $Z$, it follows that $Tp(\ddt|_t) = 0 \in T_0^H(\R^2_S)$ for all $t \in Z$, since $p$ locally factors through $\R^0$.
Therefore, the composite $\alpha \circ Tp \circ q$ (which makes sense) is zero and hence smooth.
\end{proof}

\begin{lem}\label{lem:half-const}
For $p : \R \to \R^2_S$ smooth with
$p(t) = 0$ for $t \leq 0$, we have $p_*(\ddt|_0) = 0$ in $T_0 ^{\dvs}(\R^2_S)$.
\end{lem}

\begin{proof}
The assumption on $p$ implies that $p$ is flat at $0$, i.e., all derivatives of $p$ at $0$ are $0$.
Then by iterated use of L'H\^{o}pital's rule, one can see that $p(t^{1/3})$ is smooth.
So the plot $p$ factors through the smooth map $t^3 : \R \to \R$ which sends $\ddt|_0$ to $0$.
\end{proof}

\begin{ex}\label{ex:spaghetti-r}
Let $x = 0 \in \R^2_S$.
We will construct a smooth linear map $\alpha : T_0^{\dvs}(\R^2_S) \to \R$ that does not
extend to a smooth fibrewise linear map from $T^{\dvs} \R^2_S$ to $\R$.
In other words, we will show that the inclusion $\Uplambda^1 \R^2_S \to (T^{\dvs}\R^2_S)^*$
described in Remark~\ref{rem:Uplambda-Exterior} is proper.
Using the notation of Section~\ref{sss:E'},
this is saying that the inclusion $E' \to E$ is proper, where
$E \to \R^2_S$ is the dual of the tangent bundle.

For each $k = 1, 2, 3, \ldots$, consider the line $p_k : \R \to \R^2_S$ sending $t$ to $(t, kt)$
and let $p_k'(0) = (p_k)_*(\ddt|_0)$ be its tangent vector at $0$.
Let $p : \R \to \R^2_S$ be a smooth curve such that
$p(t) = 0$ for $t \leq 0$,
$p(1/k) = 0$ for each $k$,
and $p'(1/k)$ is a non-zero multiple of $p_k'(0)$ for each $k$.
Such a curve exists by considering points $\pm(2^{-k}, k2^{-k})$ and applying
the Special Curve Lemma~\cite[Corollary~2.10]{KM}.
Note that the first condition implies that $p'(0) = 0$, using Lemma~\ref{lem:half-const}.

Now, $T_0 ^{\dvs}(\R^2_S)$ is a fine diffeological vector space (Proposition~\ref{prop:fine}) and the vectors $p_k'(0)$
are linearly independent in $T_0^{\dvs} (\R^2_S)$ (\cite[Example~3.22(1)]{CW2}), so there is a smooth linear map
$\alpha : T_0^{\dvs}(\R^2_S) \to \R$ such that $\alpha(p'(1/k)) = 1$ for each $k$.
Consider an extension $\bar{\alpha} : T^{\dvs} \R^2_S \to \R$ of $\alpha$ to a fibrewise linear map.
The curve $p$ induces a smooth map $Tp : T\R \to T^{\dvs} \R^2_S$,
and we have that
$\bar{\alpha}(Tp(\ddt|_{1/k})) = \alpha(p'(1/k)) = 1$
and $\bar{\alpha}(Tp(\ddt|_0)) = \alpha(0) = 0$,
so $\bar{\alpha}$ is not smooth.
\end{ex}

We deduce that the tangent bundle of $\R^2_S$ is not trivial.

\begin{rem}
For $x, y \in \R^2_S$, there is a canonical linear diffeomorphism $T_x (\R^2_S) \cong T_y (\R^2_S)$.
This is clear because translation by $y - x$, which sends $x$ to $y$, is a diffeomorphism.
This might suggest that $T^{\dvs}\R^2_S$ is a trivial bundle, but
Example~\ref{ex:spaghetti-r} and Proposition~\ref{prop:trivial-cotangent}
imply that it is not.
This subtlety arises essentially because for fixed $a$,
the map $t \mapsto t + a$ from $\R^2_S$ to $\R^2_S$ is smooth,
but the addition map $\R^2_S \times \R^2_S \to \R^2_S$ is not smooth.
\end{rem}

\vspace*{10pt}
\end{document}